\documentclass[a4paper, 11pt]{article}
\usepackage[mathscr]{eucal}
\usepackage{amssymb}
\usepackage{latexsym}
\usepackage{amsthm}
\usepackage{mathrsfs}
\usepackage{amsmath}
\usepackage[dvips]{graphicx}
\usepackage{psfrag}
\usepackage{a4wide}
\usepackage{lscape}

\usepackage{xcolor}
\usepackage{soul}

\newcommand{\A}{\boldsymbol{A}}
\newcommand{\B}{\boldsymbol{B}}
\newcommand{\E}{\boldsymbol{E}}
\newcommand{\J}{\boldsymbol{J}}
\newcommand{\I}{\boldsymbol{I}}

\newcommand{\T}{\boldsymbol{T}}

\newtheorem{theorem}{Theorem}[section]
\newtheorem{proposition}[theorem]{Proposition}
\newtheorem{lemma}[theorem]{Lemma}
\newtheorem{corollary}[theorem]{Corollary}
\theoremstyle{definition}

\newtheorem{problem}[theorem]{Problem}
\newtheorem{remark}[theorem]{Remark}

\makeatletter
\@addtoreset{equation}{section}

\makeatother

\title{On the spectrum and linear programming bound for hypergraphs}  

\author{
Sebastian M. Cioab\u{a}, Jack H. Koolen, Masato Mimura,  \\
 Hiroshi Nozaki, and Takayuki Okuda
}
\begin{document}
\maketitle

\renewcommand{\thefootnote}{\fnsymbol{footnote}}

\begin{abstract}
The spectrum of a graph is closely related to many graph parameters. In particular, the spectral gap of a regular graph which is the difference between its valency and second eigenvalue, is widely seen an algebraic measure of connectivity and plays a key role in the theory of expander graphs. In this paper, we extend previous work done for graphs and bipartite graphs and present a linear programming method for obtaining an upper bound on the order of a regular uniform hypergraph with prescribed distinct eigenvalues. Furthermore, we obtain a general upper bound on the order of a regular uniform hypergraph whose second eigenvalue is bounded by a given value. Our results improve and extend previous work done by Feng--Li (1996) on Alon--Boppana theorems for regular hypergraphs and by Dinitz--Schapira--Shahaf (2020) on the Moore or degree-diameter problem. We also determine the largest order of an $r$-regular $u$-uniform hypergraph with second eigenvalue at most $\theta$ for several parameters $(r,u,\theta)$. In particular, orthogonal arrays give the structure of the largest hypergraphs with second eigenvalue at most $1$ for every sufficiently large $r$. Moreover, we show that a generalized Moore geometry has the largest spectral gap among all hypergraphs of that order and degree.   
\end{abstract}

\section{Introduction}

Let $H=(V,E)$ be a hypergraph, where $V$ is a finite set of vertices, and $E$ is a collection of subsets of $V$, called (hyper)edges. Suppose each element of $E$ has size at least $2$. The hypergraph $H$ is {\it $u$-uniform} if each edge contains exactly $u$ vertices and is {\it $r$-regular} if each vertex is contained in precisely $r$ edges. An $r$-regular $2$-uniform hypergraph $H$ is an $r$-regular graph. The {\it adjacency matrix} $\A$ of a hypergraph $H$ is the square matrix whose rows and columns are indexed by $V$ with entries 
\[
\A_{x,y} =
\begin{cases}
|\{ e \in E \colon\, \{x,y\} \subset e \}| \text{ if $x\ne y$}, \\
0 \text{ if $x=y$}. 
\end{cases}
\]
The eigenvalues of $\A$ are called the {\it eigenvalues} of $H$. There are other ways to define eigenvalues for hypergraphs, for example, as the eigenvalues for the adjacency hypermatrix (adjacency tensor) (see \cite{FW95,Lim05,Qi5,CD12}) or the $(u-2)$-th Laplacian matrix of the simplicial complex consisting of all subsets of edges of a $u$-uniform hypergraph (see \cite[Section 3.12]{BH12} or \cite{DR02,HJ13,PR17,BGP19}).
If $|V|=n$, we denote the eigenvalues of $H$ by $\tau_1\geq \tau_2\geq \dots \geq \tau_n$. 
If $H$ is an $r$-regular $u$-uniform hypergraph, then $\tau_1=r(u-1)$. The {\it spectral gap} $\Delta(H)$ is $\tau_1-\tau_2=r(u-1)-\tau_2$. The {\it Cheeger constant} or {\em expansion constant} $h(H)$ of the hypergraph $H$, is defined as
\[
h(H)=\min_{S\subset V, |S| \leq |V|/2} \frac{|\partial S|}{|S|},
\]
where $\partial S=\{e \in E \mid e \cap S \ne \emptyset, e \cap (V \setminus S) \ne \emptyset\}$. There are close connections between the spectral graph theory and the Cheeger constant for graphs \cite{A86,AM85,D84} and for hypergraphs \cite{R09} that led to the investigation of expander graphs and hypergraphs with good expansion properties by using their second eigenvalue.

For regular graphs, the second eigenvalue has been studied by several authors and a fundamental result due to Alon--Boppana \cite{A86,AM85} implies that for $r\geq 3$ and $\theta<2\sqrt{r-1}$, there are only finitely many $r$-regular graphs with second eigenvalue at most $\theta$. For the general case of regular uniform hypergraphs, Feng and Li \cite{FL96} (see also Li and Sol\'e \cite{LS96}) showed that if $H$ is an $r$-regular $u$-uniform hypergraph whose diameter is at least $2\ell+2$, then
\begin{equation} \label{eq:FL}
\tau_2(H) \geq u-2 +2 \sqrt{q}-\frac{2\sqrt{q}-1}{\ell}, 
\end{equation}
where $q=(r-1)(u-1)$. These results imply that for any $\theta<u-2+2\sqrt{q}$, there are only finitely many $r$-regular $u$-uniform hypergraphs $H$ with $\tau_2(H)\leq \theta$. 
An $r$-regular $u$-uniform hypergraph $H$ is called {\it Ramanujan} if 
$|\tau_2(H)-(u-2)| \leq 2 \sqrt{q}$. For more information on Ramanujan hypergraphs, see \cite{B00,L04,M00,M01} or \cite[Sec 1.5.6]{Tb}. The recent breakthrough work of Marcus, Spielman and Srivastava \cite{MSS15} imply that for any fixed $r,u\geq 3$, there exist infinite families of $r$-regular $u$-uniform Ramanujan hypergraphs. The following two problems arise naturally and we study them in this paper.
\begin{enumerate}
\item Given $r,u\geq 3$ and $\theta<u-2+2\sqrt{q}$, determine the maximum number of vertices of an $r$-regular $u$-uniform hypergraph $H$ with $\tau_2(H)\leq \theta$. 
\item Given $n,r,u\geq 3$, determine the $r$-regular $u$-uniform hypergraphs $H$ with smallest $\tau_2$ among all $r$-regular $u$-uniform hypergraphs with $n$ vertices.
\end{enumerate}
These problems have been studied for regular graphs \cite{CKNV16,N15} and bipartite regular graphs \cite{CKN19}. In this paper, we  extend these results to general hypergraphs (or equivalently, biregular bipartite graphs). We also point out some connections between our work and the Moore or degree-diameter problem which asks for the maximum order of a regular graph or hypergraph of given valency and diameter. Roughly speaking, we can use our spectral results to show that if a regular graph or hypergraph has a {\em small} second eigenvalue $\tau_2$ (this is described precisely later in the paper), then it cannot have the largest order. This result is an extension of a result in \cite{DSS}. Moore polygons \cite{DG81} are optimal graphs in regard to the problems (1) and (2) for $r$-regular simple graphs as proved in \cite{CKNV16,N15}. The linear programming bound (LP bound) below and the orthogonal polynomials $(F_i(x))_{i\geq 0}$ therein played an essential role in obtaining these results. 
\begin{theorem}[LP bound \cite{N15}]
Let $G$ be a connected $r$-regular graph with order $n$ and distinct eigenvalues $\tau_1=r>\tau_2>\ldots >\tau_d$. Let $(F_i(x))_{i\geq 0}$ be the sequence of orthogonal polynomials defined by 
$F_0(x)=1$, $F_1(x)=x$, $F_2(x)=x^2-r$, and  
$
F_i(x)=x F_{i-1}(x)-(r-1)F_{i-2}(x) 
$
for $i\geq 3$. If there exists a polynomial $f(x)=\sum_{i\geq 0}f_iF_i(x)$ such that $f(r)>0$, $f(\tau_j)\leq 0$ for each $1\leq j\leq d$, $f_0>0$ and $f_i\geq 0$ for each $1\leq i\leq \deg f$, then $n \leq f(r)/f_0$. 
\end{theorem}

In the present paper, we generalize the LP bound to $r$-regular $u$-uniform hypergraphs for any degrees $r,u\geq 3$. The universal cover of a connected $r$-regular graph is the homogeneous tree of degree $r$ which is an infinite distance-regular graph \cite{HO07}. The orthogonal polynomials $(F_i(x))_{i\geq 0}$ are obtained from the intersection numbers of this distance-regular graph. The infinite distance-regular graphs are completely classified \cite{I83} (see also \cite{BI86,MW89}) and they are universal covers of connected $r$-regular $u$-uniform hypergraphs. We will obtain the linear programming bound for hypergraphs using certain orthogonal polynomials associated with these infinite distance-regular graphs.  

The present paper is organized as follows. In Section~\ref{2}, we describe the basic objects studied in this paper and the relationships between biregular graphs and regular uniform hypergraphs. Knowing the spectrum of a bipartite biregular graph is essentially the same as knowing that of the corresponding regular uniform hypergraph. Therefore, results for regular uniform hypergraphs have a natural interpretation for bipartite biregular graphs. In Section~\ref{3}, we introduce the infinite distance-regular graphs, and present some of their properties and of their associated orthogonal polynomials. In Section~\ref{4}, we prove the linear programming bound on the order of a connected $r$-regular $u$-uniform hypergraph (Theorem \ref{thm:lp_bound}) that generalizes the LP bound above.
In Section~\ref{5}, we prove that if an $r$-regular $u$-uniform hypergraph of girth $g$ and diameter $d$ with precisely $d+1$ distinct eigenvalues satisfies $g\geq 2d-1$, then the point graph of the hypergraph is distance-regular. If $g\geq 2d$, then the hypergraphs are known as generalized Moore geometries (see Section~\ref{5} for more details). In Section~\ref{6}, we obtain an upper bound on the order of a hypergraph with given second eigenvalue (Theorem \ref{thm:largest_graph}), that generalizes the result in \cite{CKNV16} for the problem (1). We also show that the hypergraphs attaining equality in Theorem \ref{thm:largest_graph} are generalized Moore geometries. Theorem \ref{thm:largest_graph} also improves the contrapositive of the above result \eqref{eq:FL} from Feng and Li \cite{FL96}. It also extends an upper bound obtained by Dinitz, Schapira and Shahaf \cite{DSS} related to the degree-diameter problem. In Section~\ref{7}, we determine the exact value of the maximum order for the problem (1) for several parameters. In particular, if $r\geq \max \{7u-5,u^2-1\}$ and there exists an orthogonal array $OA(u+1,r+1)$, then we prove that $h(r,u,1)=u(r+1)$ (Theorem \ref{thm:ru1}). In Section~\ref{8},  we show that generalized Moore geometries are optimal hypergraphs in regard to the problem (2). This generalizes the result from \cite{N15}.

\section{Hypergraphs and bipartite biregular graphs} \label{2}

The statements in this section are well known (see for example \cite{LS96}). Let $B=(V_1,V_2,E)$ be a bipartite graph, where $V_1$ and $V_2$ are the color classes or partite sets of $B$. The bipartite graph $B$ is an {\it $(r,u)$-biregular graph} if each vertex in $V_1$ has degree $r$ and each vertex in $V_2$ has degree $u$. Let $\A(B)$ be the adjacency matrix of $B$, that is indexed by $V_1 \cup V_2$ with $(x,y)$-entry 1 if $\{ x,y \} \in E$ and 0 otherwise. The eigenvalues of $\A(B)$ are called the {\it eigenvalues} of $B$. 

Let $H=(V,E)$ be a hypergraph. For $x_0,\ldots,x_p \in V$ and $e_1,\ldots,e_p \in E$, a sequence \[w_p=(x_0,e_1,x_1,e_2,x_2, \ldots,e_p,x_p)\] is a {\it walk} of length $p$ from $x_0$ to $x_p$ if $x_0 \in e_1$, $x_p \in e_p$, $x_{i} \in e_{i} \cap e_{i+1}$ for each $i \in \{1,\ldots, p-1\}$, and $x_{i} \ne x_{i+1}$ for each $i \in \{0,\ldots, p-1\}$.  A subsequence $(x_{i},e_{i},x_{i+1},e_{i+1},x_{i+2})$ of a walk $w_p$  is {\it backtracking} if $e_{i}=e_{i+1}$. A walk $w_p$ is {\it non-backtracking} if $w_p$ has no backtracking or $p=1$. A walk from $x$ to $y$ is called a {\it Berge cycle} if $x=y$ and all its edges and vertices are distinct (with the exception of $x=y$). Note that for hypergraphs, there are other notions of cycles (see \cite{KMV} for example), but we will only consider Berge cycles in our paper. The minimum length of a cycle in $H$ is called the {\it girth} of $H$.  A hypergraph $H$ is {\it connected} if there exists a walk from $x$ to $y$ for any $x,y \in V$. When $H$ is connected, the {\it distance} between two vertices $x$ and $y$ is the minimum length of a walk from $x$ to $y$. The {\it diameter} of $H$ is the maximum distance between two vertices of $H$.  

Let $\A$ be the adjacency matrix of $H$. This is the matrix whose rows and columns are indexed by the vertex set $V$, where $\A_{xy}$ equals the number of edges of $H$ containing $x$ and $y$ if $x\neq y\in V$, and is $0$, when $x=y$. Parallel edges are not allowed in hypergraphs, but when $u\geq 3$, we may have two distinct vertices $x\neq y$ being contained in two different edges of $H$. In that case, $\A_{xy}>1$ and this means that the girth of $H$ is 2.

Let $\Gamma_H$ denote the multigraph whose vertex set is $V$, and the number of edges between $x$ and $y$ equals $\A_{xy}$.  The graph $\Gamma_H$ is called the {\it point graph} or {\it point multigraph} of $H$. The {\it incidence graph} or {\it Levi graph} of $H$ is the bipartite graph $(V,E,E')$, where $E'=\{ \{x,e\}  \colon\,  x \in V, e \in E, x \in e\}$. 
The {\it dual} hypergraph $H^*$ of $H$ is the hypergraph $(V^*,E^*)$, 
where $V^*=E$ and $E^*=\{ \{e \in E \colon\, x \in e \} \colon\, x \in V\}$.  

Let $H$ be an $r$-regular $u$-uniform hypergraph. 
The dual hypergraph $H^*$ is a $u$-regular $r$-uniform hypergraph. 
Let $B$ be the incidence graph of $H$, which is isomorphic to that of $H^*$. 
The bipartite graph $B$ is an $(r,u)$-biregular graph. 
We can similarly obtain a regular uniform hypergraph from a biregular graph. 
Let $\A(H)$, $\A(H^*)$ and $\A(B)$ be the adjacency matrices of 
$H$, $H^*$ and $B$, respectively. 
They are related as follows.
\[
\A(B)=\begin{pmatrix}
O & \boldsymbol{M} \\
{}^t\boldsymbol{M} & O 
\end{pmatrix},
\]
\begin{equation} \label{eq:semi-reg}
\A(B)^2=\begin{pmatrix}
\boldsymbol{M}{{}^t\boldsymbol{M}} &O \\
O & {{}^t\boldsymbol{M}}  \boldsymbol{M}
\end{pmatrix}
=
\begin{pmatrix}
\A(H) +r \I_V & O \\
O & \A(H^*) + u \I_E
\end{pmatrix},
\end{equation}
 where $\I_V$ is the identity matrix of size $|V|$. 
Let $P(x)$, $P^*(x)$ and $Q(x)$ be the characteristic polynomials of 
$\A(H)$, $\A(H^*)$ and $\A(B)^2$, respectively. It follows that $Q(x)=P(x-r)P^*(x-u)$
and $x^{|V|} P^*(x-u)=x^{|E|} P(x-r)$. This implies that the eigenvalues of $B$ are the square roots of the eigenvalues of $\A(H)$ plus $r$, and that of $\A(H^*)$ plus $u$. There are {\it obvious} eigenvalues $0$ of $B$ with multiplicity $\left| |V|-|E| \right|$. 
Hence, the spectrum of one of the three graphs $B$, $H$ and $H^*$ yields that of the other two.

\section{Infinite distance-regular graphs}  \label{3}

Let $V$ be a countable set. A simple graph $G = (V, E)$ is {\it locally finite} if the degree of each vertex is finite and is {\it connected} if between any two distinct vertices, there is at least one walk. The graphs considered in this section are infinite, locally finite and connected. Let $d(x,y)$ denote the distance from $x \in V$ to $y \in V$. Let $R_i=\{ (x,y) \in V\times V \colon\, d(x,y)=i\}$. 
The {\it $i$-th distance matrix} $\A_i$ of $G$ is the matrix indexed by $V$ whose $(x,y)$-entry is 
$1$ if $(x,y) \in R_i$, and 0 otherwise.  
The matrix $\A_1$ is the adjacency matrix since the graph is simple. 
A locally finite graph $G = (V, E)$ is a {\it distance-regular graph} if 
for any non-negative integers $i,j,k$, 
the number $p_{ij}^k=|\{z \in V \colon\, (x,z) \in R_i, (z,y) \in R_j \}|$ is independent of the choice of $(x,y) \in R_k$. 
The {\it intersection array} of a distance-regular graph is defined to be
\[
\begin{pmatrix}
*   & c_1 &c_2 & \cdots \\
a_0 & a_1 & a_2 & \cdots \\
b_0 & b_1 & b_2 & \cdots
\end{pmatrix}, 
\]
where $a_i=p_{1,i}^i$, $b_i=p_{1,i+1}^i$, $c_i=p_{1,i-1}^i$, and $\ast$ means the undefined entry $c_0$.  
Let $v_i(x)$ be the polynomial of degree $i$ defined recursively by 
\[
v_0(x)=1, \qquad v_1(x)=x, 
\]
\[
c_{i+1} v_{i+1}(x) =(x-a_i) v_i(x)-b_{i-1}v_{i-1}(x)\quad (i\geq 1).
\]
The matrix $\A_i$ satisfies $\A_i=v_i(\A_1)$ \cite[Section 4.1]{BCNb} (see also \cite{DKT16}). 

Let $s=u-1$ and $t=r-1$.
The intersection array of any infinite distance-regular graph is of the form 
\[
\begin{pmatrix}
*   & 1 &1 & \cdots \\
0 & s-1 & s-1 & \cdots \\
s(t+1) & st & st & \cdots
\end{pmatrix} 
\]
for two integers $s, t \geq 1$ (see \cite{I83}).   
There exists a unique distance-regular graph $D_{r,u}$ with the above intersection array. 
Let $T_{r,u}$ be the infinite $(r,u)$-biregular tree. 
The vertex set of $D_{r,u}$ is the partite set consisting of the vertices of degree $r$ of $T_{r,u}$, and 
two vertices $x,y$ are adjacent if their distance $d(x,y)$ in $T_{r,u}$ is 2. 
Each vertex of $D_{r,u}$ lies in the intersection of exactly 
$r$ copies of the complete graphs of order $u$. 
Note that $D_{r,u}$ is vertex-transitive, and $D_{r,2}$ is the homogeneous $r$-regular tree. 

Let $F^{(r,u)}_i(x)$ denote the polynomial of degree $i$ as $v_i(x)$ of $D_{r,u}$, 
that is defined by 
\[
F^{(r,u)}_0(x)=1, \qquad F^{(r,u)}_1(x)=x, \qquad F^{(r,u)}_2(x)=x^2-(u-2)x-r(u-1)
\]
and 
\begin{equation} \label{eq:3-term}
F^{(r,u)}_{i+1}(x)=(x-u+2) F^{(r,u)}_i(x)- (r-1)(u-1) F^{(r,u)}_{i-1}(x). 
\end{equation}
The connections between these polynomials and graphs and hypergraphs have been studied by various authors (see \cite{MW89,LS96,CX15} for example). For $q=(r-1)(u-1)$ and $k=r(u-1)$, 
one has $F^{(r,u)}_i(k)=kq^{i-1}$ for $i \geq 1$, and $F^{(r,u)}_0(k)=1$. 
Let $\delta_I(x)$ be the function defined by $\delta_I(x)=1$ if $x \in I$, and $\delta_I(x)=0$ otherwise. Let $
I_{r,u}=[u-2-2 \sqrt{q}, u-2+2\sqrt{q}]$. It is known (see \cite{MW89}) that the polynomials $F^{(r,u)}_i(x)$ form a sequence of monic orthogonal polynomials with respect to the non-negative weight
\begin{equation} \label{eq:weight1}
w(x)=
\frac{r\sqrt{4 q- (x-u+2)^2}}{2\pi (k-x)(r+x)} \delta_{I_{r,u}}
\end{equation}
if $r\geq u$, or with respect to 
\begin{equation} \label{eq:weight2}
w(x)=\frac{r\sqrt{4 q- (x-u+2)^2}}{2 \pi (k-x)(r+x)}\delta_{I_{r,u}} +\frac{u-r}{u} \delta_{\{-r\}}(x) 
\end{equation}
if $r< u$. 

Let $G_i^{(r,u)}(x)=\sum_{j=0}^i F^{(r,u)}_j(x)$. 
By the three-term recurrence relation \eqref{eq:3-term}, 
it follows that
\[
G_i^{(r,u)}(x) =\frac{F_{i+1}^{(r,u)}(x)-q F_{i}^{(r,u)}(x)}{x-k}. 
\]
The polynomials $G_i^{(r,u)}(x)$ form a sequence of orthogonal 
polynomials  with respect to the non-negative weight 
\[
(k-x) w(x),
\]
where $w(x)$ is defined in equations $\eqref{eq:weight1}$ and 
$\eqref{eq:weight2}$ (see \cite[Lemmas 3.3, 3.5]{CK07}).   

Since $D_{r,u}$ is distance-regular, $\A_i=F^{(r,u)}_i(\A_1)$ for the $i$-th distance matrix $\A_i$ of $D_{r,u}$. 
The spectrum of $\A_1$ is $I_{r,u}$ if $r\geq u$,  
and $I_{r,u} \cup \{-r\}$ if $r<u$ (see \cite{MW89}). The following result will be used later to prove Theorem \ref{thm:largest_graph}.
\begin{theorem} \label{thm:coef_posi}
Suppose $r>2$. Let $p_l(i,j)$ be the coefficients of $F^{(r,u)}_i(x)F^{(r,u)}_j(x)$ in terms of $F^{(r,u)}_l(x)$, namely 
$$F^{(r,u)}_i(x)F^{(r,u)}_j(x)=
\sum_{l=0}^{i+j}p_l(i,j)F^{(r,u)}_l(x).$$
Then $p_l(i,j)\geq 0$ for any $l,i,j$ and $p_0(i,j)=F^{(r,u)}_i(k)\delta_{ij}$, where $\delta$ is Kronecker's delta and $k=r(u-1)$. 
Moreover $p_l(i,j) > 0$ if and only if 
 $|i-j|\leq l \leq i+j$ for $u>2$, and $p_l(i,j) > 0$ if and only if $|i-j|\leq l \leq i+j$, 
and $l \equiv i+j \pmod{2}$ for $u=2$. 
\end{theorem}
\begin{proof}

We modify the proof of Theorem 5.2 in \cite{Ab}, that uses induction on $i$. We may suppose $i\leq j$. The condition $|i-j|\leq l \leq i+j$ is necessary for $p_l(i,j) > 0$ because $\langle F_i^{(r,u)}(x)F_j^{(r,u)}(x), F_l^{(r,u)}(x) \rangle=\langle F_j^{(r,u)}(x),F_i^{(r,u)}(x)F_l^{(r,u)}(x)\rangle=0$ for $l+i<j$, where $\langle,\rangle$ is the inner product of the orthogonal polynomials. 

The equation $p_0(i,j)=F^{(r,u)}_i(k)\delta_{ij}$ is clear
because $$ F^{(r,u)}_i(k)\delta_{ij}=kq^{i-1}\delta_{ij}=\langle F_i^{(r,u)}(x),F_j^{(r,u)}(x) \rangle=\langle F_i^{(r,u)}(x)F_j^{(r,u)}(x),1 \rangle=p_0(i,j).$$
We can prove the assertion follows for $i=1,2$ and $j\geq 1$ by direct calculation. Assume $F^{(r,u)}_i(x)F^{(r,u)}_j(x)$ satisfies the assertion for each $j \geq i$. 
It follows that for $3 \leq i+1\leq j$, 
\begin{align} \label{eq:2.1}
F^{(r,u)}_{i+1}(x)F^{(r,u)}_j(x)&=\left((x-u+2)F^{(r,u)}_i(x)-qF^{(r,u)}_{i-1}(x)\right)F^{(r,u)}_j(x)  \nonumber \\ \nonumber
&=F^{(r,u)}_i(x)\left(F^{(r,u)}_{j+1}(x)+(u-2)F^{(r,u)}_j(x)+qF^{(r,u)}_{j-1}(x)\right)\\ \nonumber
&\qquad -(u-2)F^{(r,u)}_i(x)F^{(r,u)}_j(x)-qF^{(r,u)}_{i-1}(x)F^{(r,u)}_j(x)\\
&=F^{(r,u)}_i(x)F^{(r,u)}_{j+1}(x)+q\left(F^{(r,u)}_i(x)F^{(r,u)}_{j-1}(x)-
F^{(r,u)}_{i-1}(x)F^{(r,u)}_j(x) \right).
\end{align}
From equation \eqref{eq:2.1}, 
 it follows that 
\begin{align*}
F^{(r,u)}_{i+1}(x)F^{(r,u)}_{j}(x)-
F^{(r,u)}_{i}(x)&F^{(r,u)}_{j+1}(x)
=q \left(F^{(r,u)}_{i}(x)F^{(r,u)}_{j-1}(x)-
F^{(r,u)}_{i-1}(x)F^{(r,u)}_{j}(x) \right)\\
&=q^2 \left(F^{(r,u)}_{i-1}(x)F^{(r,u)}_{j-2}(x)-
F^{(r,u)}_{i-2}(x)F^{(r,u)}_{j-1}(x)\right)\\
&=q^{i-1}\left(F^{(r,u)}_{2}(x)F^{(r,u)}_{j-i+1}(x)-
F^{(r,u)}_{1}(x)F^{(r,u)}_{j-i+2}(x)\right)\\
&=q^{i-1}\left(-(u-1)F^{(r,u)}_{j-i+1}(x)+q(u-2)F^{(r,u)}_{j-i}(x)+q^2F^{(r,u)}_{j-i-1}(x) \right). 
\end{align*} 
Thus,  $F^{(r,u)}_{i+1}(x)F^{(r,u)}_{j}(x) $ satisfies the assertion  from the induction assumption except for the coefficient of the term $F^{(r,u)}_{j-i+1}(x)$.  
Let $c$ be the coefficient of the term $F^{(r,u)}_{j-i+1}(x)$ in $F^{(r,u)}_{i}(x)F^{(r,u)}_{j+1}(x)$. 
From the equations
\[
c k q^{j-i} =\langle F^{(r,u)}_{i}(x)F^{(r,u)}_{j+1}(x), F^{(r,u)}_{j-i+1}(x) \rangle =
\langle F^{(r,u)}_{i}(x)F^{(r,u)}_{j-i+1}(x),F^{(r,u)}_{j+1}(x) \rangle= k q^j, 
\]
we can determine $c=q^i$. 
Therefore the coefficient of the term $F^{(r,u)}_{j-i+1}(x)$ in  $F^{(r,u)}_{i+1}(x)F^{(r,u)}_{j}(x) $  is
equal to $c-q^{i-1}(u-1)=(r-2)(u-1) q^{i-1}>0$.  Thus the theorem follows. 
\end{proof}

\section{LP bound for regular uniform hypergraphs } \label{4}

In this section, we obtain an upper bound on the order of an $r$-regular $u$-uniform hypergraph with  given  distinct eigenvalues (Theorem \ref{thm:lp_bound}) which we call the (hypergraph) linear programming bound (LP bound). 

When $u\geq 3$, note that a non-backtracking walk in a hypergraph $H$ is not the same as a non-backtracking walk in its associated multigraph $\Gamma_H$. When $u=3$ for example, consider an edge $e_1=\{x_0,x_1,x_2\}$ of a $3$-uniform hypergraph $H$. The walk $(x_0,e_1,x_1,e_1,x_2)$ is a walk with backtracking in the hypergraph $H$ while the same walk $(x_0,x_1,x_2)$ is a non-backtracking walk in the multigraph $\Gamma_H$. Therefore, when counting non-backtracking walks in $u$-uniform hypergraphs with $u\geq 3$, our following theorem is a new result that does not follow by simply applying the graph theoretic results on the non-backtracking walks in $\Gamma_H$.

\begin{theorem} \label{thm:num_non-back}
If $H$ is a connected $r$-regular $u$-uniform hypergraph 
with adjacency matrix $\A$, 
then the $(x,y)$-entry of $F_i^{(r,u)}(\A)$ is the number of 
non-backtracking walks of length $i$ from $x$ to $y$. 
\end{theorem}
\begin{proof}
We use induction on $i$. 
Let $b_{xy}^{(i)}$ be the number of non-backtracking walks of length $i$ from $x$ to $y$. 
Let $f_{xy}^{(i)}$ be the $(x,y)$-entry of $F_i^{(r,u)}(\A)$. 
For $i=1$, the assertion is trivial. 
Next we consider the case $i=2$. The $(x,y)$-entry $a_{xy}^{(2)}$ of $\A^2$ is the number of walks of length 2 from $x$ to $y$. 
A walk of length 2 from $x$ that has backtracking is of the form $(x,e,x_1,e,y)$. 
For $x=y$,  the number of walks $(x,e,x_1,e,x)$ is equal to $r(u-1)$ for each $x\in V$,  
because one has $r$ choices of $e$ and $u-1$ choices of $x_1$ in $e$. 
For $x \ne y$, the number of  walks $(x,e,x_1,e,y)$ is equal to $(u-2)\A_{xy}$ for any $x,y \in V$, because one has $\A_{xy}$ choices of $e$ ($x,y \in e$) and $u-2$ choices of $x_1$ $(x_1\in e, x_1\ne x,y)$.  
The assertion therefore follows from $b_{xy}^{(2)}=a_{xy}^{(2)}-(u-2)\A_{xy}-r(u-1)\delta_{xy}=f_{xy}^{(2)}$, where
$\delta$ is the Kronecker delta. 

Suppose $i\geq 2$ and $b_{xy}^{(j)}=f_{xy}^{(j)}$ for each 
$j \in \{1,\ldots, i\}$.  
Because $$F^{(r,u)}_{i+1}(\A)=\big(\A-(u-2)\I \big) F^{(r,u)}_i(\A)- (r-1)(u-1) F^{(r,u)}_{i-1}(\A),$$ we have that
\begin{align*}
f_{xy}^{(i+1)}&=\sum_{z \in V} (f_{xz}^{(1)} -(u-2)\delta_{xz})f_{zy}^{(i)}-(r-1)(u-1)f_{xy}^{(i-1)}\\
&= \sum_{z \in V} b_{xz}^{(1)}b_{zy}^{(i)} -\sum_{z \in V}(u-2)\delta_{xz}b_{zy}^{(i)}-(r-1)(u-1)b_{xy}^{(i-1)}\\
&= \sum_{z \in V} b_{xz}^{(1)}b_{zy}^{(i)} -(u-2)b_{xy}^{(i)}-(r-1)(u-1)b_{xy}^{(i-1)}
\end{align*}
The value $\sum_{z \in V} b_{xz}^{(1)}b_{zy}^{(i)}$ is the number of 
walks $(x,e_1,z,e_2,x_2,\ldots, e_{i+1},y)$ such that  the subsequence $(z,e_2,x_2, \ldots, e_{i+1},y)$ is a non-backtracking walk. We count the number of backtracking walks there.  
If the walk $(x,e_1,z,e_2,x_2,\ldots, e_{i+1},y)$ has backtracking, then $e_1=e_2$ holds.  
Suppose $x \in e_3$. 
For each non-backtracking walk $(x,e_3,x_3\ldots, e_{i+1},y)$, the number of choices of $(z,e_2)$ is $(r-1)(u-1)$
because one has $r-1$ choices of $e_2$ $(e_2\ne e_3)$ and $u-1$ choices 
of $z$ $(z \ne x)$ in $e_2$. This case corresponds to $(r-1)(u-1)b_{xy}^{(i-1)}$.  
Suppose $x \not\in e_3$. 
For each non-backtracking walk $(x,e_2,x_2,\ldots, e_{i+1},y)$, 
one has $u-2$ choices of $z$ $(z\ne x,x_2)$ such that 
$(x,e_2,z,e_2,x_2,\ldots, e_{i+1},y)$ is non-backtracking.   
 This case corresponds to $ (u-2)b_{xy}^{(i)}$.
  Therefore $b_{xy}^{(i+1)}=f_{xy}^{(i+1)}$ follows. 
\end{proof}
The following corollaries are immediate.
\begin{corollary} \label{coro:non-negativeA}
If $H$ is a connected $r$-regular $u$-uniform hypergraph 
with adjacency matrix $\A$, 
then each $(x,y)$-entry of $F_i^{(r,u)}(\A)$ is non-negative for each $i \geq 0$. 
\end{corollary}
\begin{corollary}
If $H$ is a connected $r$-regular $u$-uniform hypergraph 
with adjacency matrix $\A$, 
then the following are equivalent. 
\begin{enumerate}
\item ${\rm tr}(F_i^{(r,u)}(\A))=0$ for each $i$ with $1 \leq i \leq g-1$, and ${\rm tr}(F_g^{(r,u)}(\A)) \ne 0$. 
\item The girth of $H$ is $g$. 
\end{enumerate}
\end{corollary}

The following result is called the LP bound for regular uniform hypergraphs. It generalizes previous works from \cite{CKNV16,CKN19}.
\begin{theorem} \label{thm:lp_bound}
Let $H$ be a connected $r$-regular $u$-uniform hypergraph 
on $v$ vertices. 
Let $\tau_0,\tau_1,\ldots, \tau_d$ 
be the distinct eigenvalues of $H$, where $\tau_0=r(u-1)=k$. 
If there exists a polynomial $f(x)=\sum_{i\geq 0} f_i F_i^{(r,u)}(x)$
such that $f(k)>0$, $f(\tau_j) \leq 0$ for each $j \in \{1,\ldots, d\}$, 
$f_0>0$, and $f_i\geq 0$ for each $i \in \{ 1, \ldots, \deg f\}$, then 
\[
v\leq \frac{f(k)}{f_0}. 
\]
Equality holds if and only if 
$f(\tau_j)=0$ for each $j \in \{1,\ldots,d\}$ and 
 $f_i\cdot{\rm tr}(F_i^{(r,u)}(\A))=0$ for each $i \in \{1,\ldots, \deg f\}$. 
\end{theorem}
\begin{proof}
Let $\A$ be the adjacency matrix of $H$. 
The matrix $\A$ can be expressed by the spectral decomposition $\A=\sum_{i=0}^d \tau_i \E_i$, where $\E_0=(1/v) \J, \E_1,\dots,\E_d$ are the primitive idempotents (projections onto the eigenspaces) of $A$.   
Applying $f$ to both sides, we obtain that
\begin{align} \label{eq:lp1}
\sum_{j=0}^d f(\tau_j) \E_j= f(\A) =\sum_{i\geq 0} f_i F_i^{(r,u)}(\A).
\end{align}
Taking traces in \eqref{eq:lp1} and using our hypothesis, we get that
\[
f(k)={\rm tr}(f(k) \E_0) \geq {\rm tr}\left( \sum_{j=0}^d f(\tau_j) \E_j \right) 
={\rm tr} \left( \sum_{i\geq 0} f_i F_i^{(r,u)}(\A)  \right) 
\geq {\rm tr} (f_0 \I) =v f_0,   
\]
and hence $v\leq f(k)/f_0$. 
The equivalent condition for this equality is immediate.  
\end{proof}
\begin{remark}
If a polynomial $f(x)$ satisfies the condition from Theorem~\ref{thm:lp_bound}, then so does $f(x)/f_0$. 
Thus we may normalize $f_0=1$. 
Theorem~\ref{thm:lp_bound} can be expressed as 
the following linear programming problem and its dual. 
\[
v \leq \max_{m_i} \left\{1+m_1+\cdots+m_d \colon\, 
\begin{array}{cc}
-\sum_{i=1}^d m_i F_j^{(r,u)}(\tau_i)\leq F_j^{(r,u)}(k), & j=1,\ldots,s, \\
m_i\geq 0, & i=1,\ldots ,d 
\end{array}
\right\},
\]
\[
v \leq \min_{f_j} \left\{ 1+f_1F_1^{(r,u)}(k)+\cdots+f_sF_s^{(r,u)}(k) \colon\,
\begin{array}{cc}
-\sum_{j=1}^s f_j F_j^{(r,u)}(\tau_i) \geq 1, & 
i=1,\ldots,d,\\
 f_j\geq 0, &j=1,\ldots, s
\end{array}
\right\},
\]
where $s$ is the degree of $f$, $m_i$ is the multiplicity of $\tau_i$ and $f_0=1$. 
For given $s$, we can estimate this bound numerically.  
\end{remark}

\begin{remark} \label{rem:girth}
Suppose $f(x)$ satisfies the condition from Theorem~\ref{thm:lp_bound} and $f_i>0$ for each $i \in \{1,\ldots, \deg f\}$. 
Then the equality of the bound holds if and only if $f(\tau_j)=0$ for each $j\in \{1,\ldots d\}$, and the girth of $H$ is 
at least $\deg f +1$. 
\end{remark}

\begin{remark}
For an $(r,u)$-biregular graph $B=(V_1,V_2,E)$, it is possible to apply LP bounds for the corresponding hypergraph $H$. 
From Theorem~\ref{thm:lp_bound}, we can obtain a bound as follows
\[
|V_1|+|V_2| = |V_1| + \frac{r}{u}|V_1| \leq \frac{r+u}{u}\cdot \frac{f(k)}{f_0}. 
\]
\end{remark}

\section{Distance-regularity for graphs with large girth} \label{5}

Let $H$ be an $r$-regular $u$-uniform hypergraph of 
adjacency matrix $\A$. Let $\Gamma_H$ be the point graph of $H$. 

\begin{theorem} \label{thm:drg}
Let $H$ be a connected $r$-regular $u$-uniform hypergraph of girth $g$, 
with only $d+1$ distinct eigenvalues. 
If $g \geq 2d-1 \geq 3$ holds, then 
the point graph $\Gamma_H$ is a distance-regular graph of diameter $d$.
\end{theorem}
\begin{proof}
Since the girth $g$ is at least 3, $\Gamma_H$ is a simple graph. 
Let $\A_i$ be the $i$-th distance matrix of $\Gamma_H$. 
From $g \geq 2d-1$, we have 
\begin{equation} \label{eq:Ai}
\A_i=F_i^{(r,u)}(\A_1)
\end{equation} for 
$0 \leq i \leq d-1$, because the local structure of $\Gamma_H$ around a vertex is the same as $D_{r,u}$.  
Let $\B_d$ be the matrix defined by $\B_d=\J-\sum_{i=0}^{d-1}\A_i$. 
Let $\mathfrak{A}$ be the linear space spanned by $\{\A_0, \A_1, \ldots , \A_{d-1}, \B_d\}$ over $\mathbb{C}$. 
Note that $\mathfrak{A}$ is closed under the Hadamard product, and $\I \in \mathfrak{A}$. 

Let $\{\tau_0,\tau_1,\ldots, \tau_d\}$ be 
the eigenvalues of $\Gamma_H$, where $\tau_0=k=r(u-1)$.  
The matrix  $\A_1$ can be expressed by the spectral decomposition $\A_1=\sum_{i=0}^d \tau_i \E_i$, where 
$\E_0=(1/v)\J$ and $v$ is the order of $\Gamma_H$. 
We can express that 
\begin{align*}
\A_i&=F_i^{(r,u)}(\A_1)=\sum_{j=0}^d F_i^{(r,u)}(\tau_j) \E_j \qquad (0\leq i \leq d-1), \\
\B_d&= v \E_0 -\sum_{i=0}^{d-1}\A_i. 
\end{align*}
Thus $\mathfrak{A}$ is spanned by $\{\E_0,\ldots, \E_d\}$. 
It follows from $\dim \mathfrak{A}=d+1$ 
that $\{\E_0, \ldots, \E_d\}$ is a basis of $\mathfrak{A}$. 
Therefore $\mathfrak{A}$ is closed under the ordinary matrix multiplication, and $\J \in \mathfrak{A}$. 

By Theorem 2.6.1 in \cite{BCNb}, 
$\mathfrak{A}$ is the Bose-Mesner algebra of an association scheme of class $d$. Moreover, from equations \eqref{eq:Ai} and 
\[
\B_d =v \prod_{i=1}^d \frac{\A_1-\tau_i \I}{k-\tau_i} - 
\sum_{i=0}^{d-1} F_i^{(r,u)} (\A_1) ,
\]
the association scheme is $P$-polynomial \cite[Section 3.1]{BIb}, namely $\Gamma_H$ is 
a distance-regular graph of diameter $d$.  
\end{proof}
\begin{remark} \label{rem:gen_Moore}
Let $s=u-1$ and $t=r-1$. The intersection array of $\Gamma_H$ in Theorem~\ref{thm:drg} has the form {\small 
\[
\begin{pmatrix}
* & 1  & \cdots & 1 &1& c_2 \\
0 & s-1 & \cdots &s-1& s-1+c_1 & s(t+1) -c_2 \\
s(t+1) & st &  \cdots &st&st-c_1 & *
\end{pmatrix}
\]
}for some integers $c_1\geq 0$, $c_2 \geq 1$.  
If $g\geq 2d$ holds, then $c_1=0$ and $H$ is a generalized Moore geometry \cite[see p. 236 (iii)']{BIb}.   
An $r$-regular $u$-uniform hypergraph $H$ is called a {\it generalized Moore geometry} 
if in the point graph $\Gamma_H$ of diameter $d$, any two distinct vertices with distance less than $d$ are joined by a unique shortest path, and any two distinct points with distance $d$ are joined by exactly $c_2$ distinct shortest paths.  Roos and Van Zanten \cite{RvZ82} introduced the concept of generalized Moore geometries. When $c_2=1$, these are the same as Moore geometries who were studied by Bose and Dowling \cite{BD71} as generalizations of Moore graphs introduced by Hoffman and Singleton \cite{HS60}.
\end{remark}

\section{Maximizing the order for given second eigenvalue} \label{6}

Throughout Sections~\ref{6} and \ref{7}, the hypergraphs $H$ considered are connected with girth at least $3$, meaning that the point graph $\Gamma_H$ is a simple connected graph. Let $h(r,u,\theta)$ denote the maximum order  of a connected $r$-regular $u$-uniform hypergraph 
whose second eigenvalue is at most $\theta$. 
In this section, we give an upper bound for $h(r,u,\theta)$ and characterize
the hypergraphs attaining the bound. 
First we prove several results that are used later. 
\begin{lemma} \label{lem:T}
Let $\T(r,u,d,c)$ be the $(d+1) \times (d+1)$ tridiagonal matrix 
with diagonals 
\[
\begin{bmatrix}
* & 1 & \cdots & 1 & c \\
0 & s-1 & \cdots& s-1 & s(t+1)-c \\
s(t+1) & st & \cdots &st & *
\end{bmatrix},
\] 
where $s=u-1$ and $t=r-1$. 
Then the characteristic polynomial of $\T(r,u,d,c)$ is 
$$(x-k)(\sum_{i=0}^{d-1} cF_i^{(r,u)}(x)+F_{d}^{(r,u)}(x)),$$
where $k=r(u-1)$. 
\end{lemma}
\begin{proof}
By strong induction on $i$, one can prove that $F_i^{(r,u)}(x)$ is the characteristic polynomial of the principal submatrix of $\T(r,u,i+1,c)$ obtained by deleting the last row and column of this matrix. The Laplace expansion along the last row of $\T(r,u,d,c)$ yields 
\begin{align*}
|x \I -&\T(r,u,d,c)|
=(x-k+c)F_d^{(r,u)}(x)-cqF_{d-1}^{(r,u)}(x)\\
&=(x-k)F^{(r,u)}_d(x)+c(F_d^{(r,u)}(x)-qF_{d-1}^{(r,u)}(x))\\
&=(x-k)F_d^{(r,u)}(x)+c(x-k)F_{d-1}^{(r,u)}(x)+c(F_{d-1}^{(r,u)}(x)-qF_{d-2}^{(r,u)}(x))\\
&=(x-k)\left(F_d^{(r,u)}(x)+c\sum_{i=0}^{d-1}F_i^{(r,u)}(x)  \right), 
\end{align*}
where $k=s(t+1)$ and $q=st$. 
\end{proof}

\begin{lemma} \label{lem:5.2}
Let $\lambda_j$ be the largest zero of $G_j^{(r,u)}(x)$. 
Then $\bigcup_{j=1}^\infty [ \lambda_j, \lambda_{j+1}] = [-1, u-2 +2 \sqrt{q})$ holds, where $q=(r-1)(u-1)$. 
\end{lemma}
\begin{proof}
The largest zero $\lambda_j$ is  located between $u-2+  2\sqrt{q} \cos (\pi/j)$ and $u-2+ 2\sqrt{q} \cos (\pi/(j+1))$ (see \cite{FL96}). 
This implies that $\lim_{j \to \infty} \lambda_j=u-2 +2 \sqrt{q}$. Since $\lambda_1 = -1$, the assertion follows.   
\end{proof}

The following is the main theorem in this section. 
\begin{theorem} \label{thm:largest_graph}
Let $k=r(u-1)$ and $q=(r-1)(u-1)$. 
Let $\lambda_j$ be the largest zero of $G_j^{(r,u)}(x)$. Let 
$\theta$ be a real number in $[-1, u-2 +2 \sqrt{q})$. 
There exists an integer $d\geq 1$ such that $\lambda_{d-1} < \theta \leq  \lambda_{d}$ by Lemma~\ref{lem:5.2}.  Then 
\[
h(r,u,\theta) \leq 1+\sum_{j=0}^{d-2} kq^j+ \frac{kq^{d-1}}{c},
\]
where  
$c=-F_{d}^{(r,u)}(\theta)/G_{d-1}^{(r,u)}(\theta)$. 
Equality holds if and only if there is a generalized Moore geometry whose intersection array is the tridiagonal entries of $\T(r,u,d,c)$. 
\end{theorem}
\begin{proof}
Since $c=-F_{d}^{(r,u)}(\theta)/G_{d-1}^{(r,u)}(\theta)$ and 
$F_{d}^{(r,u)}(x)=G_{d}^{(r,u)}(x)-G_{d-1}^{(r,u)}(x)$, it follows that $g(\theta)=0$, where
\begin{align} \label{eq:ch_drg}
g(x)&=(c-1)G_{d-1}^{(r,u)}(x)+G_{d}^{(r,u)}(x) \nonumber \\
&= \sum_{i=0}^{d-1}c F_{i}^{(r,u)}(x)+F_{d}^{(r,u)} (x). 
\end{align}
By $\lambda_{d-1} < \theta \leq \lambda_{d}$, we have $G_{d-1}^{(r,u)}(\theta)>0$ and $G_{d}^{(r,u)}(\theta)\leq 0$. 
This implies that $c\geq 1$ and $\theta$ is the largest zero of $g(x)$. 
Let $f(x)$ be the polynomial defined to be 
\[
f(x)=\frac{g(x)^2}{x-\theta}. 
\]
We would like to apply Theorem~\ref{thm:lp_bound} to $f(x)$. It is clear that the value $f(\tau)$ is not positive for any $\tau \leq \theta$. 

We prove that $f(x)$ satisfies the condition of coefficients from Theorem~\ref{thm:lp_bound}. 
By Theorem 3.1 in \cite{CK07}, 
the polynomial $g(x)/(x-\theta)$
has positive coefficients in terms of $G_i^{(r,u)}(x)$ ($0\leq i\leq d-1$).
This implies that $g(x)/(x-\theta)$ 
has positive coefficients in terms of $ F_i^{(r,u)}(x)$ ($0 \leq i \leq d-1$). 
By Theorem~\ref{thm:coef_posi}, 
$f(x)$ has positive coefficients 
in terms of $F_{i}^{(r,u)}(x)$ ($0 \leq i \leq 2d-1$). 
Thus $f(x)$ satisfies the condition from Theorem~\ref{thm:lp_bound}. 

Set the positive coefficients $f_i$, $g_i$, and $h_i$ as 
$f(x)=\sum_{i=0}^{2d-1}f_iF_{i}^{(r,u)}(x)$, 
$g(x)=\sum_{i=0}^{d}g_iF_{i}^{(r,u)}(x)$,
and $h(x)=g(x)/(x-\theta)=\sum_{i=0}^{d-1}h_iF_{i}^{(r,u)}(x)$. 
By Theorem~\ref{thm:coef_posi},  it follows that 
\[
f_0=\sum_{i=0}^{d-1}g_ih_iF_{i}^{(r,u)}(k)
=\sum_{i=0}^{d-1}ch_iF_{i}^{(r,u)}(k)=ch(k).
\]
Applying Theorem~\ref{thm:lp_bound}
to $f(x)$, one has
\begin{align*}
h(r,u,\theta)& \leq \frac{f(k)}{f_0}=\frac{g(k)}{c}\\
&=\sum_{i=0}^{d-1}F_{i}^{(r,u)}(k)+\frac{F_{d}^{(r,u)}(k)}{c}\\
&=1+\sum_{j=0}^{d-2} kq^j+ \frac{kq^{d-1}}{c}. 
\end{align*}
If the equality holds, then the girth of an extremal hypergraph $H$ is
at least $2d$ by Remark~\ref{rem:girth}. Since the number of zeros of $f(x)$ is $d$, 
the number of non-trivial distinct eigenvalues of $H$ is at most $d$. 
Therefore $H$ is a generalized Moore geometry whose intersection array is the tridiagonal entries of $\T(r,u,d,c)$ by Theorem~\ref{thm:drg}, Remark~\ref{rem:gen_Moore}, equation \eqref{eq:ch_drg}, and Lemma~\ref{lem:T}.  
It is clear that a generalized Moore geometry attains the bound. 
\end{proof}
\begin{remark}
The upper bound $M(r,u,\theta)=1+\sum_{j=0}^{d-2} kq^j+ kq^{d-1}/{c}$ is monotonically increasing on $\theta \in [-1,u-2+2\sqrt{q})$ since $c\geq 1$ and $c=-F_{d}^{(r,u)}(\theta)/G_{d-1}^{(r,u)}(\theta)$ is monotonically decreasing on $\theta \in (\lambda_{d-1},\lambda_{d}]$. 
\end{remark}
The following is immediate from Theorem~\ref{thm:largest_graph}. 
\begin{corollary}[Feng--Li \cite{FL96}] \label{coro:infi}
There exist only finitely many $r$-regular $u$-uniform hypergraphs 
whose second eigenvalues are less than  $u-2+2\sqrt{q}$. 
\end{corollary}
\begin{remark}
For $(r,u)$-biregular graphs $B$, it follows that 
\[
\lim \inf \tau_2(B) \geq \sqrt{r-1}+\sqrt{u-1}
\]
as the number of vertices in $B$ tends to infinity, 
where $\tau_2(B)$ is the second eigenvalue of $B$ \cite{LS96}.  We can obtain similar results to Theorem~\ref{thm:largest_graph} and Corollary~\ref{coro:infi} for biregular graphs. 
\end{remark}
Feng and Li \cite{FL96} prove that if an $r$-regular $u$-uniform hypergraph $H$ has diameter at least $2\ell+2$, then $\tau_2(H)\geq u-2 +2 \sqrt{q}-(2\sqrt{q}-1)/\ell$. 
The contrapositive of this result is that 
if $\tau_2(H) < u-2 +2 \sqrt{q}-(2\sqrt{q}-1)/\ell$ holds, then the diameter of $H$ is at most $2\ell+1$, and hence the order is at most $1+\sum_{j=0}^{2 \ell} kq^j$. Our next result improves this bound and extends a result from \cite{Ci20} to hypergraphs. 
\begin{theorem}\label{thm:upperbndn}
Suppose $r\geq 3$, $u\geq 2$ and $\ell \geq 1$.  Let $H$ be a connected $r$-regular $u$-uniform hypergraph of order $v$. If the second eigenvalue $\tau_2$ of $H$ is at most $u-2 +2 \sqrt{q}-(2\sqrt{q}-1)/\ell$, then $v\leq 1+\sum_{j=0}^{2\ell-1} kq^j$.  
The equality holds only if $(r,u,\ell)=(3,2,1)$, and the Petersen graph attains the bound.    
\end{theorem}
\begin{proof}
Recall that $\lambda_j$ is the largest zero of $G_j^{(r,u)}(x)$. 
For $(r,u,\ell)=(3,2,1)$, we have $\tau_2 \leq 1$ from our assumption. 
Actually $\tau_2=1$ and $v\leq 1+\sum_{j=0}^{1} kq^j=10$ by Theorem~\ref{thm:largest_graph}. Equality is attained by the Petersen graph. 
It is enough to prove that $u-2 +2 \sqrt{q}-(2\sqrt{q}-1)/\ell < \lambda_{2\ell}$ except for $(r,u,\ell)=(3,2,1)$ by Theorem~\ref{thm:largest_graph}. 
For $\ell =1$, we have $u-2 +2 \sqrt{q}-(2\sqrt{q}-1)/\ell = u-1$.  Since $G_2^{(r,u)}(u-1)=-(r-2)(u-1)+1<0$ for $(r,u)\ne (3,2)$, we have $u-1<\lambda_2$. 
From $\lambda_{2 \ell} >u-2+ 2\sqrt{q} \cos (\pi/2\ell)$ \cite{FL96},  
we prove $f(\ell)=u-2+ 2\sqrt{q} \cos (\pi/2\ell)-(u-2 +2 \sqrt{q}-(2\sqrt{q}-1)/\ell)=2\sqrt{q}(\cos (\pi/2\ell)-1+1/\ell)-1/\ell>0$ for $\ell \geq 2$. 
The derivative satisfies 
\[
\frac{df}{d \ell} =
2 \sqrt{q} \left( \frac{\pi}{2 \ell^2} \sin \frac{\pi}{2\ell}- \frac{1}{\ell^2} \right)+\frac{1}{\ell^2}
<  2 \sqrt{q} \left( \frac{\pi}{2 \ell^2} \sin \frac{\pi}{2\ell}- \frac{1}{2 \ell^2} \right)<0
\]
for $\ell \geq 5$. Since $f(\ell)\geq 2\sqrt{2} (\cos(\pi /2 \ell)-1+1/\ell)-1/\ell >0$ for $2 \leq \ell  \leq 4$ and $\lim_{\ell \to \infty} f(\ell)=0$, it follows that $f(\ell )>0$ for $\ell \geq  2$. 
\end{proof}  
Dinitz, Schapira, and Shahaf \cite{DSS} improved the Moore bound from 
a certain condition of non-trivial eigenvalues as follows. 
\begin{theorem}[Dinitz--Schapira--Shahaf \cite{DSS}] \label{thm:DSS}
Let $\Gamma$ be a connected $r$-regular graph of diameter $d$ with $n$ vertices.
Then, for each non-trivial eigenvalue $\tau$ of $\Gamma$:
$$
|G_{d}^{(r,2)}(\tau)|\leq 1+\sum_{i=0}^{d-1}r(r-1)^i-n=G_d^{(r,2)}(r)-n.
$$
\end{theorem}

\begin{corollary}[Dinitz--Schapira--Shahaf \cite{DSS}]  \label{coro:DSS}
If $\Gamma$ is an $r$-regular graph of diameter $2$ with $n$ vertices, then the second eigenvalue $\tau_2$ satisfies
\begin{equation} \label{eq:111} 
\tau_2 \leq \frac{-1+\sqrt{1+4(r^2+r-n)}}{2},
\end{equation}
or equivalently
\[
n \leq r^2+r-\tau_2^2-\tau_2=G_2^{(r,2)}(r)-G_2^{(r,2)}(\tau_2). 
\]
\end{corollary} 
Note that the inequality \eqref{eq:111} that is in Theorem 2 \cite{DSS} has a typo at $-1$.  For Corollary~\ref{coro:DSS}, if $G_2^{(r,2)}(\tau_2)<0$ holds, then the bound is worse than the Moore bound $n\leq G_2^{(r,2)}(r)$. Our next results extends Theorem \ref{thm:DSS} to hypergraphs.
\begin{theorem}\label{thm:DSSgen}
Let $H$ be an $r$-regular $u$-uniform hypergraph of diameter $d$ with $n$ vertices.
Then, each non-trivial eigenvalue $\lambda$ of $H$ satisfies 
\[
\left| G_d^{(r,u)}(\lambda) \right|
\leq 1+\sum_{i=0}^{d-1} kq^i- n,
\]
where $k=r(u-1)$ and $q=(r-1)(u-1)$.
\end{theorem} 
\begin{proof}
The adjacency matrix $\A$ of $H$ can be expressed by the spectral decomposition 
$k\E_k+ \sum_{\lambda \in \Lambda(H)}\lambda \E_{\lambda}$, where $\Lambda(H)$ be the set of the 
non-trivial eigenvalues of $H$ and $\E_k=(1/n)\J $. 
Then, the matrix 
\[
G_d^{(r,u)}(\A)-\J=(G_d^{(r,u)}(k)-n)\E_k+ \sum_{\lambda \in \Lambda(H)}G_d^{(r,u)}(\lambda) \E_{\lambda}
\] 
is non-negative. By the Perron--Frobenius theorem, one has $|G_d^{(r,u)}(\lambda)|\leq G_d^{(r,u)}(k)-n=1+\sum_{i=0}^{d-1} kq^i-n$. 
\end{proof}
Theorem \ref{thm:DSSgen} implies the upper bound $n\leq G_d^{(r,u)}(k) -| G_d^{(r,u)}(\lambda) |$ for any eigenvalue $\lambda$ of $H$. When $\lambda=\tau_2(H)$, we can improve the bound as follows. 
\begin{theorem}\label{coro:imp2}
Let $H$ be an $r$-regular $u$-uniform hypergraph of diameter $d$ with $n$ vertices.
Let $\tau_2$ be the second eigenvalue of $H$, and $\lambda_i$ be the largest zero of $G_i^{(r,u)}(x)$ for $i\geq 0$. 
Let $k=r(u-1)$ and $q=(r-1)(u-1)$. If $\tau_2 \geq \lambda_d$, then
\[
n \leq G_d^{(r,u)}(k)-G_d^{(r,u)}(\tau_2),
\]
where $G_d^{(r,u)}(\tau_2)\geq 0$. 
If $\lambda_{d-1}<\tau_2< \lambda_d$, then
\[
n \leq 1+\sum_{j=0}^{d-2} kq^j+ \frac{kq^{d-1}}{c} \leq G_d^{(r,u)}(k)+G_d^{(r,u)}(\tau_2),
\]
where  
$c=-F_{d}^{(r,u)}(\tau_2)/G_{d-1}^{(r,u)}(\tau_2)$ and $G_d^{(r,u)}(\tau_2)<0$. If $\tau_2\leq \lambda_{d-1}$ holds, then $n \leq 1+\sum_{j=0}^{d-2} kq^j$. 
Moreover, if $\lambda_{d-1}<\tau_2< \lambda_d$ and $q\geq 6$, then
$ 1+\sum_{j=0}^{d-2} kq^j+ \frac{kq^{d-1}}{c} < G_d^{(r,u)}(k)+G_d^{(r,u)}(\tau_2)
$. 
\end{theorem}
\begin{proof}
For the case $\lambda \geq  \lambda_d$, 
the assertion is immediate by Theorem~\ref{thm:DSSgen}. 
For the case $\lambda < \lambda_d$, the upper bounds are obtained from Theorem~\ref{thm:largest_graph}. 
We prove that if $\lambda_{d-1}<\tau_2< \lambda_d$ holds, then $ 1+\sum_{j=0}^{d-2} kq^j+ \frac{kq^{d-1}}{c} \leq G_d^{(r,u)}(k)+G_d^{(r,u)}(\tau_2)$, moreover we prove this equality does not happen when $q\geq 6$. 
We can calculate
\begin{align*}
G_d^{(r,u)}(k)+G_d^{(r,u)}(\tau_2)- (1+\sum_{j=0}^{d-2} kq^j+ \frac{kq^{d-1}}{c})&=kq^{d-1}+G_d^{(r,u)}(\tau_2)+\frac{kq^{d-1}G_{d-1}^{(r,u)}(\tau_2)}{F_{d}^{(r,u)}(\tau_2)}\\
&=G_d^{(r,u)}(\tau_2)\left(1+\frac{kq^{d-1}}{F_{d}^{(r,u)}(\tau_2)}\right).
\end{align*}
Since $F_{d}^{(r,u)}(\tau_2)<0$, we would like to show $|F_{d}^{(r,u)}(\tau_2)|\leq kq^{d-1}=F_{d}^{(r,u)}(k)$. 

Let $\A$ be the adjacency matrix of $H$. 
By Corollary~\ref{coro:non-negativeA}, 
each entry of $F_{d}^{(r,u)}(\A)$ is non-negative, and hence 
$|F_{d}^{(r,u)}(\tau_2)|\leq F_{d}^{(r,u)}(k)$ by Perron--Frobenius theorem. 
For $q\geq 6$, we prove  $|F_{d}^{(r,u)}(\tau_2)|< kq^{d-1}$ by induction on $d$. 
For $d=1$, $|F_{1}^{(r,u)}(\tau_2)|=|\tau_2| < k$. 
For $d=2$, $|F_{2}^{(r,u)}(\tau_2)|\leq |\tau_2(\tau_2-u+2)|+|k| \leq 2 \sqrt{q}k+k< k q$ for $q\geq 6$. 
For $d\geq 2$, suppose $|F_{i}^{(r,u)}(\tau_2)|< kq^{i-1}$ for each $i \in \{1,\ldots,d\}$. 
Then it follows that 
\begin{align*}
F_{d+1}^{(r,u)}(\tau_2)&=(\tau_2-u+2)F_{d}^{(r,u)}(\tau_2)-qF_{d-1}^{(r,u)}(\tau_2)\\
&\leq|\tau_2-u+2|\cdot |F_{d}^{(r,u)}(\tau_2)|+|qF_{d-1}^{(r,u)}(\tau_2)|\\
&< 2\sqrt{q} k q^{d-1}+ kq^{d-1}<kq^{d}
\end{align*}
for $q\geq 6$. Therefore the theorem follows. 
\end{proof}
An $r$-regular $u$-uniform hypergraph of diameter $d$ has a natural upper bound $n \leq 1+\sum_{j=0}^{d-1} kq^j$ on the order $n$. The difference $(1+\sum_{j=0}^{d-1} kq^j)-n$ is called the {\it defect} of a hypergraph. 
The following is immediate from Theorem \ref{coro:imp2}.

\begin{corollary}\label{coro:defect}
Let $H$ be an $r$-regular $u$-uniform hypergraph of diameter $d$ with $n$ vertices.
Let $\tau_2$ be the second eigenvalue of $H$, and $\lambda_i$ be the largest zero of $G_i^{(r,u)}(x)$. 
Let $e$  be the defect of $H$. Then it follows that 
\begin{enumerate}
\item $e\geq G_d^{(r,u)}(\tau_2)$ if $\tau_2 \geq \lambda_d$,  \label{en:1}
\item $e\geq kq^{d-1}G_{d}^{(r,u)}(\tau_2)/F_{d}^{(r,u)}(\tau_2)$ if $\lambda_{d-1}<\tau_2< \lambda_d$, \label{en:2}
\item $e\geq kq^{d-1}$ if $\tau_2\leq \lambda_{d-1}$. 
\end{enumerate}
\end{corollary}
Informally, the corollary says that if $|\lambda_d-\tau_2|$ is large, then the defect becomes large. For the case of diameter $d=2$, the inequalities   \eqref{en:2} and \eqref{en:1} in Corollary~\ref{coro:defect} can be expressed by 
\begin{equation}
\tau_2 \geq \frac{u-2-K+\sqrt{(u-K)^2+4q}}{2} \label{eq:lower_2}
\end{equation}
and
\begin{equation}
\tau_2 \leq  \frac{u-3+\sqrt{(u-1)^2+4q+4e}}{2} , \label{eq:upper_2}
\end{equation}
respectively, where $K=kq/(kq-e)$. If the defect is smaller than given value $e<kq^{d-1}$, then $\tau_2$ must be in the region between the values of \eqref{eq:lower_2} and \eqref{eq:upper_2}. Table~\ref{tb:1-0} shows the order $v$ and the defect $e$ for several largest known $r$-regular graphs of diameter $d$, and the region of the second eigenvalue $\tau_2$. 
The value $\lambda_2$ in Table \ref{tb:1-0} represents the second eigenvalue of a $r$-regular Moore graph of diameter $d$. 
Another interpretation of these results is for example, that if $(r,d)=(8,2)$, then any $8$-regular graph $H$ whose second eigenvalue $\tau_2(H)$ is either less than $2.09503$ or larger than $3.40512$, must have less than $1+\sum_{i=0}^{d-1}r(r-1)^i-e=65-8=57$ vertices. 

\begin{table}
\begin{center}

\quad 

\begin{tabular}{|c||c|c|c|c|c|} \hline
$(r,d)$ & $v$ & $e$ & \eqref{eq:lower_2} & $\lambda_2$ & \eqref{eq:upper_2} \\ \hline \hline 
(8,2)&57&8 &2.09503& 2.19258& 3.40512 \\ \hline
(9,2)&74&8&2.29956&2.37228& 3.53113 \\ \hline
(10,2)&91&10 &2.46923& 2.54138& 3.88748 \\ \hline
(4,3) & 41 &12&2.11232& 2.25342& 2.88396 \\ \hline
(5,3) &72 &34 &2.42905& 2.62620& 3.77862 \\ \hline
(4,4) &98 &63 &2.53756& 2.69963& 3.44307 \\ \hline
(5,4) &212&214 &2.91829& 3.12941& 4.41922 \\ \hline 
(3,5) & 70&24&2.32340& 2.39309& 2.64401 \\ \hline 
(4,5) & 364 &121&2.89153& 2.93996& 3.42069 \\ \hline
(3,6) & 132 &58&2.45777& 2.51283& 2.75001 \\ \hline
(4,6)& 740 &717&3.00233& 3.08314& 3.73149 \\ \hline
\end{tabular}
\end{center}
\caption{Regions of $\tau_2$ for given defect} \label{tb:1-0}
\end{table}

\section{Exact values on $h(r,u,\theta)$} \label{7}

Let $h(r,u,\theta)$ ({\it resp.} $v(k,\theta)$, $b(k,\theta)$) be the largest possible order of an $r$-regular $u$-uniform hypergraph ({\it resp.} $k$-regular graph, $k$-regular bipartite graph) whose second eigenvalue is at most $\theta$. In this section, we determine several exact values of $h(r,u,\theta)$, $v(k,\theta)$ and $b(k,\theta)$. We collect these results for small values of $r,u,\theta$ in Tables \ref{tab:u=2} and \ref{tab:u=3} in the Appendix section. We also prove $h(r,u,1)\leq u(r+1)$ if $r\geq \max\{7u-5,u^2-1\}$. 

\begin{lemma} \label{lem:m-s}
For $r,u\geq 2$, $r\cdot h(r,u,\theta)=u\cdot h (u,r,\theta+r-u)$. 
Moreover, a hypergraph $H$ attains $h(r,u,\theta)$ if and only if  the  dual hypergraph $H^*$ attains $h(u,r,\theta+r-u)$.   
\end{lemma}
\begin{proof}
Let $H$ be an $r$-regular $u$-uniform hypergraph with second eigenvalue $\tau_2$. 
From equation~\eqref{eq:semi-reg}, the dual hypergraph $H^*$ has 
second eigenvalue $\tau_2+r-u$. From the equation 
$r|H|=u|H^*|$, this lemma follows.  
\end{proof}
We always suppose $r\geq u$ from Lemma~\ref{lem:m-s} throughout this section. 
For a finite set $V$ of size $v$, a family $B$ of $k$-element subsets of $V$ is a {\it combinatorial $t$-$(v,k,\lambda)$ design} if any $t$-element subset of $V$ is contained in exactly $\lambda$ subsets of $B$. 
Here an element of $V$ is called a {\it point} and an element of $B$ is called a ${\it block}$ of the design. 

\begin{lemma}
Let $k=r(u-1)$. 
There exists a combinatorial $2$-$(k+1,u,1)$ design if and only if there exists an $r$-regular $u$-uniform hypergraph whose second eigenvalue is in $[-1,0)$. 
Moreover, if the design exists, then $h(r,u,\theta)=k+1$ for $-1 \leq \theta <0$. 
\end{lemma}
\begin{proof}
If there exists a combinatorial $2$-$(k+1,u,1)$ design, 
then we can obtain an $r$-regular $u$-uniform hypergraph $(V,B)$ with second eigenvalue $-1$ by 
interpreting blocks as hyperedges. 
If there exists an $r$-regular $u$-uniform hypergraph $H$ whose second eigenvalue is in $[-1,0)$, the point graph is the complete $k$-regular graph \cite{S70}. 
Since any two distinct vertices in $H$ are contained in only one edge, the edges can be interpreted as the blocks of the design. 
\end{proof}

\begin{lemma}
If a bipartite $r$-regular graph attaining $b(r,\theta)$ has girth at least 6, then $h(r,r,\theta^2-r)=b(r,\theta)/2$. 
\end{lemma}
\begin{proof}
From a bipartite $r$-regular graph $H$ of girth at least $6$ attaining $b(r,\theta)$, we construct an $r$-regular $r$-uniform hypergraph whose incidence/Levi graph is $H$. The rest follows from the results at the end of Section \ref{2}.
\end{proof}

\begin{lemma}\label{lem:count_tri}
Let $H$ be an $r$-regular $u$-uniform hypergraph of order $v$. 
Then $rv/u$ is an integer. 
\end{lemma}
\begin{proof}
The number of edges in $H$ is $rv/u$, which is an integer. 
\end{proof}

The next result shows that it is possible to further improve Theorem \ref{thm:largest_graph} (which gives the upper bound $v(6,2)\leq 47.7$).
\begin{lemma} \label{lem:v(6,2)}
$v(6,2)\leq 45$. 
\end{lemma}
\begin{proof}
Applying the LP bound for
\[
f(x)=F^{(6,2)}_4(x)+9F^{(6,2)}_3(x)+\frac{121}{4}F^{(6,2)}_2(x)+64F^{(6,2)}_1(x)+\frac{153}{2}F^{(6,2)}_0(x), 
\]
we obtain $v(6,2)\leq 136/3=45.333...$
\end{proof}

\begin{proposition}[{Corollary 3.12.3 in \cite{BCNb}}] \label{prop:min-2}
A connected regular graph whose smallest eigenvalue is greater than $-2$ is a complete graph or an odd cycle. 
\end{proposition}

\begin{corollary} \label{cor:<1}
For $k\geq 3$ and $\theta \in [0,1)$, $v(k,\theta)=2k$, and the graph of order $2k$ is the complete bipartite graph.  
\end{corollary}
\begin{proof}
If $\Gamma$ is a $k$-regular graph whose second eigenvalue is less than $1$, then the smallest eigenvalue of the complement $\Gamma^c$ of $\Gamma$ is greater than $-2$. 
By Proposition~\ref{prop:min-2}, $\Gamma^c$ is a union of complete graphs or odd cycles. 
Note that $\Gamma^c$ is a $(|\Gamma|-k-1)$-regular graph. If  $\Gamma^c$ contains an odd cycle as a component, then the degree of $\Gamma^c$ is 2, and the order of $\Gamma$ is $k+3$. 
If the degree of $\Gamma^c$ is greater than 2, then $\Gamma^c$ is the union of $s$ complete graphs of equal order. 
Then the order of $\Gamma$ is $sk/(s-1)$, that is maximum at $s=2$. 
Therefore $v(k,\theta)=2k$, which is attained by the complete bipartite graph for $k\geq 3$. 
\end{proof}

Let $OA(u,r)$ be an orthogonal array, which is a $u\times r^2$ array with entries from an $r$-set $S$ satisfying that in any two rows, each pair of symbols of $S$ occurs exactly once (see \cite{ACD06} or \cite[Section 10.4]{GR01} for more details).

\begin{lemma}\label{lem:oa}
If there exists $OA(u,r)$, then $h(r,u,0)=ru$. 
\end{lemma}
\begin{proof}
Let $H$ be an $r$-regular $u$-uniform hypergraph whose second eigenvalue is at most 0. 
By Smith's result \cite{S70}, the point graph of $H$ is the complete $y$-partite $x(y-1)$-regular graph 
$K(x,y)$, that satisfies each part has $x$ vertices. 
One has $r(u-1)=x(y-1)=xy-x=|H|-x$, which implies the order of $H$ is larger when $x$ is larger or $y$ is smaller.  Since the point graph of $H$ has a clique of size $u$, we have 
$y\geq u$. 
 If we find the structure of the hypergraph $H$ on $K(r,u)$, we can determine $h(r,u,0)=ru$. 
If $OA(u,r)$ exists, we can find $H$ on $K(r,u)$ as follows. 
Let $S=\{1,\ldots,r\}$ be the $r$-set as the entries of $OA(u,r)$. An entry $j\in S$ in the $i$-th row of $OA(u,r)$ can be interpreted as the $j$-th vertex in the $i$-th partite set, and  a column of $OA(u,r)$ can be interpreted as a hyperedge. Therefore we can determine $h(r,u,0)=ru$. 
\end{proof}

\begin{lemma}\label{lem:h_<1}
If $u\geq 3$ and there exists $OA(u,r)$, then $h(r,u,\theta)=ru$ for $\theta \in [0,1)$. 
\end{lemma}
\begin{proof}
Let $H$ be an $r$-regular $u$-uniform hypergraph whose second eigenvalue is at most $\theta$. 
Let $\Gamma_H^c$ be the complement of the point graph $\Gamma_H$ of $H$. 
The smallest eigenvalue of $\Gamma_H^c$ is larger than $-2$. 
By Corollary~\ref{cor:<1} and Lemma~\ref{lem:oa}, $ru \leq |\Gamma_H| \leq 2r(u-1)-1$ because the complete bipartite graph has no triangle. 

Suppose $r\geq 4$. 
 From $|\Gamma_H| \geq ru$,  
 the degree of $\Gamma_H^c$ is $|\Gamma_H|-r(u-1)-1 \geq r -1 \geq 3$. 
By Proposition~\ref{prop:min-2}, $\Gamma_H^c$ is the union of the complete graphs of equal degree, and $\Gamma_H$ is the complete multipartite graph, whose second eigenvalue is $0$.  
By Lemma~\ref{lem:oa}, $h(r,u,\theta)=ru$ for $\theta \in [0,1)$. 

If $r=3$ and $u=3$, then $h(r,u,\theta)<12$ for $\theta \in [0,1)$ by Corollary~\ref{cor:<1}.
If $|\Gamma_H|=11$ ({\it resp}. $|\Gamma_H|=10$), then $\Gamma_H^c$ is the union 
of the complete $4$-regular ({\it resp.} $3$-regular) graphs, which contradicts $|\Gamma_H|/5 \in \mathbb{Z}$ ({\it resp.} $|G|/4 \in \mathbb{Z}$). Therefore $h(3,3,\theta)=9$ for $\theta \in [0,1)$, which is attained by the hypergraph obtained from $OA(3,3)$.   
Therefore this lemma follows. 
\end{proof}

\begin{lemma} \label{lem:oa^-}
If there exists $OA(u+1,r+1)$, then 
$h(r,u,1)\geq u(r+1)$. 
\end{lemma}
\begin{proof}
An orthogonal array $OA(u+1,r+1)$ can be interpreted as   
 $u-1$ mutually orthogonal latin squares (MOLS) on $r+1$ symbols \cite{ACD06}. 
Namely, the $(r+1)^2$ columns of $OA(u+1,r+1)$ are $\boldsymbol{a}_{ij}={}^t(i,j, L_1(i,j),\ldots, L_{u-1}(i,j))$ for $1\leq i \leq r+1$ and 
$1\leq j \leq r+1$, where the matrices $L_1,\ldots, L_{u-1}$ are $u-1$ MOLS on $r+1$ symbols. 
Let a symbol $\ell$ be fixed. 
Let $I_\ell$ be the set of indices $(i,j)$ of $L_{u-1}$ such that $L_{u-1}(i,j)=\ell$.  Note that $|I_\ell|=r+1$. 
Then $T_{k,\ell}=\{L_k(i,j) \mid (i, j) \in I_\ell\}$ is a transversal of $L_k$ for $1 \leq k \leq u-2$, namely $T_{k,\ell}$ has all $r+1$ symbols \cite{W07}. 
Let $A$ be the orthogonal array $OA(u,r+1)$ obtained from $OA(u+1,r+1)$ by removing the last row $L_{u-1}$.  
Let $H$ be an $(r+1)$-regular $u$-uniform hypergraph obtained from $A$ by the manner of Lemma~\ref{lem:oa}, that has $u(r+1)$ vertices. 
Removing the edges corresponding to $\boldsymbol{a}_{ij}$ with $(i,j) \in I_\ell$ from $H$, we can obtain an $r$-regular $u$-uniform hypergraph $H'$ of order $u(r+1)$. 
The point graph of $H'$ is the complement of the $u\times (r+1)$ grid graph $K_{u} \square K_{r+1}$, where $\square$ is the Cartesian product. Thus we can show that the second eigenvalue of $H'$ is 1. This implies the lemma.  
\end{proof}
For $u=2$, there exists $OA(3,r+1)$ for any $r$. 
Then the hypergraph obtained from $OA(3,r+1)$ by the manner of Lemma~\ref{lem:oa^-} is $K_{r+1,r+1}^-$, that is the complete $(r+1)$-regular bipartite graph $K_{r+1,r+1}$ minus a perfect matching. We have $h(r,2,1)=v(r,1)=2(r+1)$ for $r \geq 11$ \cite{CKNV16}. 

\begin{lemma} \label{lem:(4,3,1)}
$h(4,3,1)=15$. 
\end{lemma}
\begin{proof} 
By Lemma~\ref{lem:count_tri}, the order $v$ of a $4$-regular $3$-uniform hypergraph $H$ is divisible by 3. 
The point graph $\Gamma_H$ of $H$ is $8$-regular and its second eigenvalue is at most $1$. 
By Theorem \ref{thm:largest_graph}, it follows that $h(4,3,1)\leq 21$ and the graph attaining the bound is a generalized quadrangle of order $(2,3)$ \cite[Section 6.5]{BCNb}. Since the hypergraph does not exist \cite[Lemma 1.15.1]{BCNb}, we have $h(4,3,1)\leq 18$. 

Suppose $v=18$. The complement $\Gamma_H^c$ of $\Gamma_H$ is a $9$-regular graph whose smallest eigenvalue is at least $-2$. By Theorem 3.12.2 in \cite{BCNb}, possible graphs for $\Gamma_H^c$ are the line graphs of $(a,b)$-biregular graphs $B$.  
We have $a+b-2=9$ and $v$ is divisible by both $a$ and $b$.  
The possible graph of $B$ is only the complete bipartite graph $K_{2,9}$. Then $\Gamma_H$ should be $K_{9,9}^{-}$, which has no triangle. Thus $h(4,3,1)= 15$, which is attained by the hypergraph obtained from $OA(4,5)$ \cite[Theorem 3.38]{ACD06} by the manner of Lemma~\ref{lem:oa^-}.  
\end{proof}

For $u=2$, $h(r,2,1)=v(r,1)$ is determined for any $r$ \cite{CKNV16}. 
For $u\geq 3$,  we prove the following as the main theorem in this section. 
\begin{theorem}\label{thm:ru1}
 Let $u \geq 3$. 
If $r\geq \max \{7u-5,u^2-1 \}$, then
 $h(r,u,1) \leq u(r+1)$. 
Moreover, if there exists an orthogonal array $OA(u+1,r+1)$, then $h(r,u,1) = u(r+1)$. 
\end{theorem}
\begin{proof}
Assume $u(r+1)< h(r,u,1)$.  
Let $H$ be an $r$-regular $u$-uniform hypergraph with second eigenvalue at most 1 and order $h>u(r+1)$.  
Let $\Gamma_H$ be the point graph of $H$, which is $r(u-1)$-regular. 
The complement $\Gamma_H^c$ of $\Gamma_H$ has smallest eigenvalue at least $-2$, and the degree is $h-r(u-1)-1$. 
By Theorem 3.12.2 in \cite{BCNb}, $\Gamma_H^c$ is one of the following $k$-regular graphs of order $v$: a subgraph of $E_7(1)$ with $v=2(k+2)\leq 28$, a subgraph of the Schl\"{a}fli graph with $v=3(k+2)/2\leq 27$, a subgraph of the Clebsch graph with $v=4(k+2)/3 \leq 16$, the line graph of a regular graph, or the line graph of a biregular graph. 

Suppose $\Gamma_H^c$ is a subgraph of $E_7(1)$. Then $h=2(h-r(u-1)+1)$, and hence $h=2r(u-1)-2>28$ from $r(u-1)\geq (u^2-1)(u-1)\geq 16$, which is a contradiction. 

Suppose $\Gamma_H^c$ is a subgraph of the Schl\"{a}fli graph. Then $h=3(h-r(u-1)+1)/2$, and hence $h=3r(u-1)-3>27$, which is a contradiction. 

Suppose $\Gamma_H^c$ is a subgraph of the Clebsch graph. Then $h=4(h-r(u-1)+1)/3$, and hence $h=4r(u-1)-4>16$, which is a contradiction. 

Suppose $\Gamma_H^c$ is the line graph of a regular graph $\Gamma'$. 
The degree of  $\Gamma'$ is $(h-r(u-1)+1)/2$. Since the number of the edges in $\Gamma'$ is $h=(h-r(u-1)+1)|\Gamma'|/4$, we have   
\[
|\Gamma'|=4+\frac{4r(u-1)-4}{h-r(u-1)+1} < 4+\frac{4r(u-1)-4}{u(r+1)-r(u-1)+1} =4u-4\frac{u^2}{r+u+1} < 4u 
\] 
from $h > u(r+1)$. Since $|\Gamma'| \leq 4u-1$ and the degree satisfies $(h-r(u-1)+1)/2 \leq |\Gamma'|-1=4u-2$, we have 
$h \leq u(r+1)-r+7u-5 \leq u(r+1)$ from $r \geq 7u-5$, which is a contradiction.  

Suppose $\Gamma_H^c$ is the line graph of a $(k,l)$-biregular graph $\Gamma'$.  The degree of  $\Gamma_H^c$ is $k+l-2=h-r(u-1)-1$, and hence 
$k+l=h-r(u-1)+1$. We may suppose $k \leq l$, namely $l\geq (h-r(u-1)+1)/2$. The number of the edges of $\Gamma'$ is $k t_1=l t_2=h$, where $t_1$ and $t_2$ are the sizes of the partite sets. 
It follows from $l\geq (h-r(u-1)+1)/2$ that 
\begin{multline*}
k \leq t_2=\frac{h}{l}\leq \frac{2h}{h-r(u-1)+1}
= 2+\frac{2r(u-1)-2}{h-r(u-1)+1}\\ < 2+\frac{2r(u-1)-2}{u(r+1)-r(u-1)+1} 
=2u-2\frac{u^2}{r+u+1}<2u,
\end{multline*}
and hence $k\leq 2u-1$. Since $l=h-r(u-1)+1-k \geq h-r(u-1)+1-(2u-1)=
h-r(u-1)-2(u-1)$, we have 
\begin{multline*}
k \leq t_2=\frac{h}{l}\leq \frac{h}{h-r(u-1)-2(u-1)}
= 1+\frac{r(u-1)+2(u-1)}{h-r(u-1)-2(u-1)}\\ < 1+\frac{r(u-1)+2(u-1)}{u(r+1)-r(u-1)-2(u-1)}
=u+\frac{u(u-1)}{r-u+2}\leq u+\frac{u(u-1)}{u(u-1)+1}<u+1 
\end{multline*}
from $r\geq u^2-1$, which implies $k\leq u$. 
Moreover, since $l=h-r(u-1)+1-k \geq h-r(u-1)+1-u$, we have 
\begin{multline*}
t_2=\frac{h}{l}\leq \frac{h}{h-r(u-1)-(u-1)}
= 1+\frac{r(u-1)+(u-1)}{h-r(u-1)-(u-1)}\\ 
< 1+\frac{r(u-1)+(u-1)}{u(r+1)-r(u-1)-(u-1)}=u
\end{multline*}
Thus $t_2< u$. Since a matching in $\Gamma'$ corresponds to a clique in the point graph $\Gamma_H$, we have $t_2\geq u$, that is a contradiction. Therefore $h\leq u(r+1)$.
By Lemma~\ref{lem:oa^-}, this theorem follows.    
\end{proof}

\section{Minimizing the second eigenvalue for given order} \label{8}
In this section, 
we give a lower bound on the second eigenvalue of a hypergraph for given order. 
This bound is attained by a generalized Moore geometry. 
Let 
$M(r,u,d,c)=1+\sum_{j=0}^{d-2}k q^j+kq^{d-1}/c$ for real number $c \geq 1$,  integer $d \geq 1$, $k=r(u-1)$, and $q=(r-1)(u-1)$. 
\begin{lemma} \label{lem:M}
For integers $n \geq 2$, $r\geq 2$ and $u\geq 2$,  
there exist unique $c$ and $d$ such that $n=M(r,u,d,c)$. 
\end{lemma}
\begin{proof}
For fixed $d$, $M(r,u,d,c)$ is monotonically decreasing when $c \geq 1$. Since $M(r,u,d,1)= \lim_{c\to \infty}M(r,u,d+1, c)$ and $M(r,u,1,k)=2$, the assertion is clear.  
\end{proof}
\begin{theorem}
Let $H$ be an $r$-regular $u$-uniform hypergraph with $n$ vertices and second eigenvalue $\tau_2$. 
Let $(d,c)$ be the unique pair with  $n=M(r,u,d,c)$ by Lemma~\ref{lem:M}. 
Let $\lambda(d,c)$ be the largest zero  of $\sum_{i=0}^{d-1} cF_i^{(r,u)}(x)+F_{d}^{(r,u)}(x)$.   
Then $\tau_2 \geq \lambda(d,c)$. 
Moreover $\tau_2=\lambda(d,c)$ if and only if 
$H$ is a generalized Moore geometry whose intersection array is as in Lemma~\ref{lem:T}. 
\end{theorem}
\begin{proof}
Assume $\tau_2 < \lambda(d,c)$ holds. 
There uniquely exists $(d',c')$ such that 
 $\tau_2=\lambda(d',c')$ by a similar manner to the proof of Lemma~\ref{lem:M}. 
Note that $\lambda(d',c')<\lambda(d,c)$ and $M(r,u,d',c')<M(r,u,d,c)$. 
Applying Theorem~\ref{thm:lp_bound} to \[
f(x)=\frac{\left( \sum_{i=0}^{d'-1} c'F_i^{(r,u)}(x)+F_{d'}^{(r,u)}(x) \right)^2}{x-\lambda(d',c')}, 
\]
 we have the upper bound $n \leq M(r,u,d',c')$. This contradicts $M(r,u,d',c')<M(r,u,d,c)=n$. 
Therefore $\tau_2 \geq \lambda(d,c)$ holds. 
Moreover $\tau_2 = \lambda(d,c)$ if and only if 
$H$ attains the bound from Theorem~\ref{thm:largest_graph}. 
The assertion therefore follows.  
\end{proof}

\section{Concluding remarks}
In this paper, we have established the linear programming bound for regular uniform hypergraphs (Theorem~\ref{thm:lp_bound}). 
The linear programming bound is applicable to problems relating to distinct eigenvalues, particularly to the second eigenvalue. 
The following are the problems we considered in this paper. 
\begin{problem}\label{prob:r1}
Given $r\geq 3$, $u\geq 2$ and $\theta<u-2+2\sqrt{q}$, determine the maximum number of vertices of an $r$-regular $u$-uniform hypergraph $H$ with $\tau_2(H)\leq \theta$. 
\end{problem}
\begin{problem}\label{prob:r2}
Given $n,r\geq 3$ and $u\geq 2$, determine the $r$-regular $u$-uniform hypergraphs $H$ with smallest $\tau_2$ among all $r$-regular $u$-uniform hypergraphs with $n$ vertices.
\end{problem} 
Generalized Moore geometries, which have the structure of a distance-regular graph with large girth, are optimal for both Problems~\ref{prob:r1} and \ref{prob:r2}. 
Besides generalized Moore geometries, we determined optimal graphs with several parameters for Problem~\ref{prob:r1}, in particular, orthogonal arrays are optimal. We expect more optimal graphs for the two problems. 
The classification of the optimal graphs for the problems with given parameters is interesting, that is related to finding good combinatorial structures like distance-regular graphs or combinatorial designs. 

For some parameters $r,u,\tau_2$, we can improve the general bound (Theorem~\ref{thm:largest_graph}), for example $(r,u,\tau_2)=(6,2,2)$ (Lemma~\ref{lem:v(6,2)}). We are interested in determination of optimal graphs for such a situation. 
\begin{problem}
Find parameters $(r,u,\tau_2)$ for which we can improve the upper bound from Theorem~\ref{thm:largest_graph}. 
\end{problem}
\begin{problem}
Is the second subconstituent of the Hoffman--Singleton graph (see \cite[Theorem 13.1.1]{BCNb})  an optimal graph for Problem~\ref{prob:r1} with parameters $(r,u,\tau_2)=(6,2,2)$? One has $42\leq v(6,2) \leq 45$, whose lower bound is obtained from the example. 
\end{problem}

\appendix

\section{Table on $h(r,u,\theta)$, $v(k,\theta)$ and $b(k,\theta)$}
\begin{table} 
\caption{$k$-regular graphs ($u=2$)} \label{tab:u=2}
\quad 

\begin{minipage}{0.5\hsize}
\begin{center}
\begin{tabular}{|c||c|c|c|}
\hline
 $k$ & $\theta$  & $b(k,\theta)$ & $v(k,\theta)$ \\ \hline \hline 

 3 & $-1$ & $4^{K_{4}}$ & $4^{K_{4}}$ \\ \hline
 3 & $<0$ & $4^{K_{4}}$ & $4^{K_{4}}$ \\ \hline
 3 & 0 & $6^{K_{3,3}}$ & $6^{K_{3,3}}$ \\ \hline
 3 & $<1$ & $6^{K_{3,3}}$ & $6^{K_{3,3},{\rm Cor.}\ref{cor:<1}}$ \\ \hline
 3 & 1 & $8^{K_{4,4}^-}$ & $10^{Pt}$ \\ \hline
 3 & $\sqrt{2}$ & $14^{GH}$ & $14^{GH}$ \\ \hline
 3 & $\sqrt{6}-1$ & $14^{GH}$ & $14^{GH}$ \\ \hline
 3 & $\sqrt{3}$ & $18^{AG}$ & $18^{\cite{CKNV16}}$ \\ \hline
 3 & $2$ & $30^{GO}$ & $30^{GO}$ \\ \hline
 3 & $\sqrt{6}$ & $126^{GD}$ & $126^{GD}$ \\ \hline \hline
 4 & $-1$  & $5^{K_{5}}$ & $5^{K_{5}}$ \\ \hline
 4 & $<0$  & $5^{K_{5}}$ & $5^{K_{5}}$ \\ \hline
 4 & 0 & $8^{K_{4,4}}$ & $8^{K_{4,4}}$ \\ \hline
 4 & $<1$  & $8^{K_{4,4}}$ & $8^{K_{4,4},{\rm Cor.}\ref{cor:<1}}$ \\ \hline
 4 & 1 & $10^{K_{5,5}^-}$ & $12^{\cite{CKNV16}}$ \\ \hline
 4 & $\sqrt{5}-1$  & $10^{K_{5,5}^-,LP}$ & $ 12^{\cite{CKNV16}}$ \\ \hline
 4 & $\sqrt{2}$& $14^{CHD}$ & $ ?\leq 19.1^{LP}$ \\ \hline
 4 & $\sqrt{3}$ & $26^{GH}$ & $26^{GH}$ \\ \hline
 4 & 2&  $ 32^{AG}$ & $35^{Od}$ \\ \hline
 4 & $\sqrt{6}$  & $80^{GO}$ & $80^{GO}$ \\ \hline
 4 & $3$& $728^{GD}$ & $728^{GD}$ \\ \hline
\hline 
 5 & $-1$ & $6^{K_{6}}$ & $6^{K_{6}}$ \\ \hline
 5 & $<0$ & $6^{K_{6}}$ & $6^{K_{6}}$ \\ \hline
 5 & $0$ & $10^{K_{5,5}}$ & $10^{K_{5,5}}$ \\ \hline
 5 & $<1$ & $10^{K_{5,5}}$ & $10^{K_{5,5},{\rm Cor.}\ref{cor:<1}}$ \\ \hline
 5 & $1$ & $12^{K_{6,6}^-}$ & $16^{Cl}$ \\ \hline
 5 & $\sqrt{3}$ & $22^{HD}$ & $?\leq 30^{LP_c}$ \\ \hline
 5 & $2$ & $42^{GH}$ & $42^{GH}$ \\ \hline
 5 & $\sqrt{5}$ & $50^{AG}$ & $?\leq 54.9^{LP}$ \\ \hline
 5 & $2\sqrt{2}$ & $170^{GO}$ & $170^{GO}$ \\ \hline
 5 & $2\sqrt{3}$ & $2730^{GD}$ & $2730^{GD}$ \\ \hline 
\end{tabular}
\end{center}
\end{minipage}
\begin{minipage}{0.5\hsize}
\begin{center}
\begin{tabular}{|c||c|c|c|}
\hline
 $k$ & $\theta$  & $b(k,\theta)$ & $v(k,\theta)$ \\ \hline \hline 

 7 & $-1$  & $8^{K_{8}}$ & $8^{K_{8}}$ \\ \hline
 7 & $<0$  & $8^{K_{8}}$ & $8^{K_{8}}$ \\ \hline
 7 & 0  & $14^{K_{7,7}}$ & $14^{K_{7,7}}$ \\ \hline
 7 & $<1$  & $14^{K_{7,7}}$ & $14^{K_{7,7},{\rm Cor.}\ref{cor:<1}}$ \\ \hline
 7 & 1  & $16^{K_{8,8}^-}$ & $18^{\cite{BCS}_{177-180}}$ \\ \hline
 7 & 2  & $ 30^{HD}$ & $50^{HoSi}$ \\ \hline
 7 & $\sqrt{7}$  & $98^{AG}$ & $? \leq 106^{LP_c}$ \\ \hline 
\hline

 9 & $-1$  & $10^{K_{10}}$ & $10^{K_{10}}$ \\  \hline
 9 & $<0$  & $10^{K_{10}}$ & $10^{K_{10}}$ \\  \hline 
 9 & $0$ & $18^{K_{9,9}}$ & $18^{K_{9,9}}$ \\  \hline 
 9 & $<1$  & $18^{K_{9,9}}$ & $18^{K_{9,9},{\rm Cor}.\ref{cor:<1}}$ \\  \hline 
 9 & $1$ & $20^{K_{10,10}^-}$ & $24^{\cite{BCS}_{183}}$ \\  \hline 
 9 & $\sqrt{3}$  & $26^{CPG}$ & $? \leq 42.7^{LP}$ \\  \hline 
 9 & $\sqrt{5}$  & $38^{HD}$ & $? \leq 68.2^{LP}$ \\  \hline 
 9 & $\sqrt{6}$  & $50^{SD_{25,9,3}}$ & $? \leq 90^{LP_c}$ \\  \hline 
 9 & $\sqrt{7}$  & $74^{SD_{37,9,2}}$ & $? \leq 116^{LP_c}$ \\  \hline 
 9 & $2\sqrt{2}$  & $146^{GH}$ & $146^{GH}$ \\  \hline 
 9 & $3$ &  $162^{AG}$ & $? \leq 178^{LP}$ \\  \hline 
 9 & $4$ &  $1170^{GO}$ & $1170^{GO}$ \\  \hline 
 9 & $2\sqrt{6}$  & $74898^{GD}$ & $74898^{GD}$ \\  \hline 
\end{tabular}
\end{center}
\end{minipage}
\end{table}

\begin{table} 
\caption{$3$-uniform hypergraphs ($u=3$)} \label{tab:u=3}
\begin{center}
\begin{tabular}{|c||c|c|c|c|c|}
\hline
$(r,u)$ & $k$ & $\theta$ & $h(r,u,\theta)$ & $b(k,\theta)$ & $v(k,\theta)$ \\ \hline \hline

$(3,3)$ & 6 & $-1$& $7^{STS}$ & $7^{K_{7}}$ & $7^{K_{7}}$ \\ \hline
$(3,3)$ & 6 & $<0$& $7^{STS}$ & $7^{K_{7}}$ & $7^{K_{7}}$ \\ \hline
$(3,3)$ & 6 & $0$& $9^{OA}$ & $12^{K_{6,6}}$ & $12^{K_{6,6}}$ \\ \hline
$(3,3)$ & 6 & $<1$& $9^{OA, {\rm Lem.}\ref{lem:h_<1}}$ & $12^{K_{6,6}}$ & $12^{K_{6,6},{\rm Cor.}\ref{cor:<1}}$ \\ \hline
$(3,3)$ & 6 & $1$& $15^{GQ}$ & $14^{K_{7,7}^-}$ & $15^{\cite{CKNV16}}$ \\ \hline
$(3,3)$ & 6 & $\sqrt{3}$& $?\leq 20.8^{LP}$ & $22^{CHD}$ & $?\leq 34.3^{LP}$ \\ \hline
$(3,3)$ & 6 & $2$& $?\leq 24^{LP,c\not\in \mathbb{Z}}$ & $32^{SD_{16,6,2}}$ & $42^{\Gamma_2(HoSi)} \leq  ?\leq 45^{{\rm Lem}.\ref{lem:v(6,2)}}$ \\\hline
$(3,3)$ & 6 & $\sqrt{5}$& $? \leq 30^{LP, no SRG}$ & $62^{GH}$ & $62^{GH}$ \\ \hline
$(3,3)$ & 6 & $3$& $63^{GH}$ & $162^{p_1}$ & $?\leq 209.7^{LP}$ \\ \hline
$(3,3)$ & 6 & $\sqrt{10}$& $? \leq 75.2^{LP}$ & $312^{GO}$ & $312^{GO}$ \\ \hline
$(3,3)$ & 6 & $\sqrt{15}$& $?\leq 283.2^{LP}$ & $7812^{GD}$ & $7812^{GD}$ \\ \hline
\hline

$(4,3)$ & 8 & $-1$ & $9^{STS}$ & $9^{K_{9}}$ & $9^{K_{9}}$ \\  \hline
$(4,3)$ & 8 & $<0$ & $9^{STS}$ & $9^{K_{9}}$ & $9^{K_{9}}$ \\  \hline
$(4,3)$ & 8 & $0$ & $12^{OA}$ & $16^{K_{8,8}}$ & $16^{K_{8,8}}$ \\  \hline
$(4,3)$ & 8 & $<1$ & $12^{OA,{\rm Lem.}\ref{lem:h_<1}}$ & $16^{K_{8,8}}$ & $16^{K_{8,8},{\rm Cor}.\ref{cor:<1}}$ \\  \hline
$(4,3)$ & 8 & $1$ & $15^{OA^-,Lem.\ref{lem:(4,3,1)}}$ & $18^{K_{9,9}^-}$ & $21^{\cite{BCS}_{181,182}}$ \\  \hline
$(4,3)$ & 8 & 2 & $? \leq 30^{LP, noSRG}$ & $30^{CHD}$ & $? \leq 51^{LP}$ \\  \hline 
$(4,3)$ & 8 & $\sqrt{7}$ & $? \leq 54^{LP,noSRG}$ & $114^{GH}$ & $114^{GH}$ \\  \hline 
$(4,3)$ & 8 & $2\sqrt{2}$ & $? \leq 69^{LP_c}$ & $128^{AG}$ & $? \leq 140.7^{LP}$ \\  \hline 
$(4,3)$ & 8 & $\sqrt{14}$ & $? \leq 180.9^{LP}$ & $800^{GO}$ & $800^{GO}$ \\  \hline 
$(4,3)$ & 8 & $\sqrt{21}$ & $? \leq 1047.6^{LP}$ & $39216^{GD}$ & $39216^{GD}$ \\  \hline 
\hline

$(5,3)$ & 10 & $-1$ & $-\infty$ & $11^{K_{11}}$ & $11^{K_{11}}$ \\  \hline
$(5,3)$ & 10 & $<0$ & $-\infty$ & $11^{K_{11}}$ & $11^{K_{11}}$ \\  \hline
$(5,3)$ & 10 & $0$ & $15^{OA}$ & $20^{K_{10,10}}$ & $20^{K_{10,10}}$ \\  \hline 
$(5,3)$ & 10 & $<1$ & $15^{OA,{\rm Lem.}\ref{lem:h_<1}}$ & $20^{K_{10,10}}$ & $20^{K_{10,10},{\rm Cor.}\ref{cor:<1}}$ \\  \hline 
$(5,3)$ & 10 & $1$ & $27^{GQ}$ & $22^{K_{11,11}^-}$ & $27^{Sch}$ \\  \hline 
$(5,3)$ & 10 & $2$ & $? \leq 39^{LP_c}$ & $32^{SD_{16,10,6}}$ & $56^{Ge}$ \\  \hline 
$(5,3)$ & 10 & $\sqrt{5}$ & $? \leq 45^{LP_c}$ & $38^{CHD}$ & $?\leq 69.2^{LP}$ \\  \hline 
$(5,3)$ & 10 & $\sqrt{7}$ & $? \leq 60^{LP_c}$ & $62^{SD_{31,10,3}}$ & $?\leq 117.4^{LP}$ \\  \hline 
$(5,3)$ & 10 & $3$ & $? \leq 90^{LP_c}$ & $182^{GH}$ & $182^{GH}$ \\  \hline 
$(5,3)$ & 10 & $3\sqrt{2}$ & $? \leq 354^{LP_c}$ & $1640^{GO}$ & $1640^{GO}$ \\  \hline 
$(5,3)$ & 10 & $3\sqrt{3}$ & $? \leq 2907^{LP_c}$ & $132860^{GD}$ & $132860^{GD}$ \\  \hline 
\end{tabular}
\end{center}
\end{table}

Tables \ref{tab:u=2} and \ref{tab:u=3} show several exact values on $h(r,u,\theta)$, $v(k,\theta)$ and $b(k,\theta)$ for $k=r(u-1)\leq 10$ and $r\geq u$. 
For $\theta \geq 1$, all $v(k,\theta)$ and $b(k,\theta)$ in the tables are already determined in \cite{CKNV16,CKN19}, and almost all graphs that attain those values are generalized Moore geometries.  
We use the following notation in the tables. 

\quad


\noindent 
$?\leq n$: the corresponding value is not determined, and it has an upper bound $n$. 

\noindent
$LP$: LP bound (Theorem~\ref{thm:largest_graph}),  \qquad 
$LP_c$: Lemma~\ref{lem:count_tri} after LP bound (Theorem~\ref{thm:largest_graph}).

\noindent
$-\infty$: non-existence, $K_{n,n}^-$: Complete bipartite graph minus a perfect matching

\noindent
$STS$: Steiner triple system $S(2,3,k+1)$,  \qquad 
$OA$: $OA(u,r)$ (Lemma~\ref{lem:oa}). 

\noindent
$OA^-$: $OA(u+1,r+1)$ 
(Lemma~\ref{lem:oa^-}). 
\qquad 
$AG$: Affine geometry 
 minus a parallel class 

\noindent
$GQ$: Generalized quadrangle
\qquad
$GH$: Generalized hexagon 
\qquad $GO$: Generalized octagon

\noindent
$\cite{BCS}_i$: Graph of no.\ $i$ in \cite{BCS}, 
\qquad 
$SD_{v,k,\lambda}$: Incidence graph of symmetric $(v,k,\lambda)$-design 

\noindent
$HD$: Incidence graph of symmetric Hadamard design 

\noindent
$CHD$: Incidence graph of the complement of symmetric Hadamard design 

\noindent
$CPG$: Incidence graph of the complement of projective plane 

\noindent
$Pt$: Petersen graph, \qquad $Od$: Odd graph, \qquad  $Cl$: Clebsch graph, 

\noindent
$p_1$; Incidence graph of partial geometry $pg(6,2,2)$, \qquad
$HoSi$: Hoffman--Singleton graph, 

\noindent
$Ge$: Gewirtz graph, \qquad 
$Sch$: Complement of th Schl\"{a}fli graph, 

\noindent
$\Gamma_2(HoSi)$: Second subconstituent of the Hoffman--Singleton graph 

\noindent
$c \not\in \mathbb{Z}$: $c\not\in \mathbb{Z}$ for LP bound (Theorem~\ref{thm:largest_graph}) 

\noindent
$noSRG$: Strongly regular graph that attains LP bound (Theorem~\ref{thm:largest_graph})  does not exist

\bigskip

\noindent
\textbf{Acknowledgments.} 
We are grateful to the two anonymous referees for their comments and suggestions that have greatly improved the first version of our paper. Cioab\u{a} has been supported by NSF grants DMS-1600768,  CIF-1815922 and a JSPS Invitational Fellowship for Research in Japan S19016. 
With the support of the grant S19016, Cioab\u{a} stayed at Aichi University of Education 
during the period from November 22, 2019 to January 20, 2020. 
Koolen is partially supported by the National Natural Science Foundation of China (No.\ 12071454), Anhui Initiative in Quantum Information Technologies (No.\ AHY150000) and the National Key R and D Program of China (No.\ 2020YFA0713100).
Mimura is partially supported by JSPS Grants-in-Aid for Scientific Research Nos.\ 17H04822 and 21K03241.
Nozaki is partially supported by JSPS Grants-in-Aid for Scientific Research Nos.\ 16K17569, 18K03396, 19K03445, and 20K03527. 
Okuda is partially supported by JSPS Grants-in-Aid for Scientific Research Nos.\ 16K17594, 16K05132, 20K14310, and 20K03589.

 \noindent
{\it Sebastian M. Cioab\u{a}} \\
Department of Mathematical Sciences,
University of Delaware, 
Newark DE 19716-2553, USA. \\
E-mail address: {\tt cioaba@udel.edu}

\quad \\

\noindent
{\it Jack H. Koolen} \\
School of Mathematical Sciences, University of Science and Technology of China, 96 Jinzhai Road, Hefei, 230026, Anhui, PR China/ 
CAS Wu Wen-Tsun Key Laboratory of Mathematics, University of Science and Technology of China, 96 Jinzhai Road, Hefei, Anhui, 230026, PR China. \\
E-mail address: {\tt koolen@ustc.edu.cn}

\quad \\

\noindent
{\it Masato Mimura}\\
Mathematical Institute, Tohoku University, 
6-3, Aramaki Aza-Aoba, Aoba-ku, Sendai 980-8578, Japan. \\
E-mail address: {\tt mimura-mas@m.tohoku.ac.jp}

\quad \\

\noindent
{\it Hiroshi Nozaki}\\ 
	Department of Mathematics Education, 
	Aichi University of Education, 
	1 Hirosawa, Igaya-cho, 
	Kariya, Aichi 448-8542, 
	Japan.\\
E-mail address: {\tt hnozaki@auecc.aichi-edu.ac.jp}

\quad \\

\noindent
{\it Takayuki Okuda}\\ 
Graduate School of Advanced Science and Engineering, Hiroshima University, 1-3-1
Kagamiyama, Higashi-Hiroshima, 739-8526, Japan \\
E-mail address:  {\tt okudatak@hiroshima-u.ac.jp}

\end{document}